\documentclass[11pt,reqno]{amsart}
\usepackage{times}
\usepackage{amsmath,amsfonts,amstext,amssymb,amsbsy,amsopn,amsthm,eucal}
\usepackage{txfonts}
\usepackage{dsfont}
\usepackage{graphicx}   % for figures
\usepackage{hyperref}
\usepackage{color}
\usepackage{verbatim}  
\usepackage{mathtools}

\usepackage[colorinlistoftodos,prependcaption]{todonotes}

\newcommand{\tr}{\text{tr}}

\newcommand{\eps}{\epsilon}
\newcommand{\Id}{\text{Id}}

\newcommand{\Rm}{\text{Rm}}
\newcommand{\Ric}{\text{Ric}}

\newcommand{\Vol}{\text{Vol}}

\newcommand{\inj}{\text{inj}}

\newcommand{\dC}{\mathds{C}}

\newcommand{\dR}{\mathds{R}}

\newcommand{\dZ}{\mathds{Z}}

\newcommand{\dist}{\text{dist}}

\newcommand{\R}{\text{R}}

\newcommand{\cB}{\mathcal{B}}

\newcommand{\cC}{\mathcal{C}}

\newcommand{\cS}{\mathcal{S}}

\newtheorem{theorem}{Theorem}[section]

\newtheorem{lemma}[theorem]{Lemma}
\newtheorem{corollary}[theorem]{Corollary}
\theoremstyle{definition}
\newtheorem{definition}[theorem]{Definition}
\theoremstyle{remark}
\newtheorem{remark}{Remark}[section]
\theoremstyle{remark}

\theoremstyle{remark}

\theoremstyle{remark}
\newtheorem{question}{Question}[section]
\theoremstyle{remark}

\theoremstyle{remark}

\begin{document}

\title{Lower Ricci Curvature and NonExistence of Manifold Structure}

\author{Erik Hupp, Aaron Naber and Kai-Hsiang Wang}

%%\date{\today}
\date{\today}
\maketitle

\begin{abstract}
It is known that a limit $(M^n_j,g_j)\to (X^k,d)$ of manifolds $M_j$ with uniform lower bounds on Ricci curvature must be $k$-rectifiable for some unique $\dim X:= k\leq n = \dim M_j$.  It is also known that if $k=n$, then $X^n$ is a topological manifold on an open dense subset, and it has been an open question as to whether this holds for $k<n$.

Consider now any smooth complete $4$-manifold $(X^4,h)$ with $\Ric>\lambda$ and $\lambda\in \dR$. Then for each $\epsilon>0$ we construct a complete $4$-rectifiable metric space $(X^4_\epsilon,d_\epsilon)$ with $d_{GH}\Big(X^4_\epsilon,X^4\Big)<\epsilon$ such that the following hold.  First, $X^4_\epsilon$ is a limit space $(M^6_j,g_j)\to X^4_\epsilon$ where $M^6_j$ are smooth manifolds with $\Ric_j>\lambda$ satisfying the same lower Ricci bound.  Additionally, $X^4_\epsilon$ has no open subset which is topologically a manifold.  Indeed, for {\it any} open $U\subseteq X^4_\epsilon$ we have that the second homology $H_2(U)$ is infinitely generated.  Topologically, $X^4_\epsilon$ is the connect sum of $X^4$ with an infinite number of densely spaced copies of $\dC P^2$ .

In this way we see that every $4$-manifold $X^4$ may be approximated arbitrarily closely by $4$-dimensional limit spaces $X^4_\epsilon$ which are nowhere manifolds.  We will see there is an, as now imprecise, sense in which generically one should expect manifold structures to not exist on spaces with higher dimensional Ricci curvature lower bounds.
\end{abstract}

\tableofcontents

\section{Introduction}\label{s:Intro}

Let us begin with a historical discussion of measured Gromov-Hausdorff limit spaces 
\begin{align}\label{e:limit_spaces}
	(M^n_j,g_j,v_j,p_j)\to (X^k,d,\nu,p)\, ,\;\;\;\Ric_j> \lambda\, ,\;\;\nu_j := \frac{dv_{g_j}}{Vol(B_1(p_j))}\, ,
\end{align}
and their structure.  That metric space limits even exist in this context was the result of Gromov \cite[Th.~5.3]{Gro99}.  Structural results for the limits $X$ began in earnest with the almost rigidities of Cheeger-Colding \cite[Th.~6.62]{CheCol96}.  Their almost splitting theorem allowed them to show that $X$ was the union of rectifiable pieces of various dimensions \cite{CheCol97}, and they conjectured that the dimension is locally constant and hence unique.  Colding-Naber resolved this conjecture in \cite{ColNab12} and showed that the limit $X^k$ is $k$-rectifiable for a unique $k$.  More recently, an example of Pan-Wei \cite{PanWei22} has shown that while $X^k$ is $k$-rectifiable, its Hausdorff dimension might be larger than $k$.  More specifically, it is possible for the singular part of $X$ to have larger dimension with respect to the Hausdorff measure than it does with respect to the limit $\nu$-measure.\\

In the context where \eqref{e:limit_spaces} is noncollapsed, which is to say $\Vol(B_1(p_j))>v>0$, one can say quite a bit more.  In this case one has that $k=n$, and by volume convergence \cite[Th.~5.9]{CheCol97} the limit measure $\nu$ is the $n$-dimensional Hausdorff measure on $X$.  The starting point for a more refined analysis of $X$ in the noncollapsed context is another almost rigidity of Cheeger-Colding \cite[Th.~4.85]{CheCol96}.  This time one considers the monotone quantity $\theta_r(x) := \frac{\Vol(B_r(x))}{\omega_n r^n}$ and shows that in the limit $X$ it is a constant in $r$ iff $B_r(x)$ is a metric cone.  This opens the door to the techniques of Federer \cite{Fed70}, which have been applied to many nonlinear equations.  In particular, one can decompose $X=\text{Reg}(X)\cup \text{Sing(X)}$ into a regular and singular part and stratify the singular part $\cS^0(X)\subseteq \cdots\subseteq \cS^{n-1}(X)=\text{Sing(X)}$.  Cheeger-Colding were able to then use the Federer dimension reduction to prove the dimensional estimates $\dim \cS^\ell \leq \ell$.  More recently, the work of Cheeger-Jiang-Naber \cite{CheJiaNab21} was able to prove that $\cS^{\ell}(X)$ is ${\ell}$-rectifiable.  This result is sharp for ${\ell}\leq n-2$ by an example of Li-Naber \cite{LiNab20}, who built examples whose singular strata $\cS^{\ell}(X)$ are ${\ell}$-rectifiable, ${\ell}$-cantor sets.\\

For the regular set $\text{Reg}(X)$ of a noncollapsed limit one can say even more.  A Reifenberg type result of Cheeger-Colding \cite[Th.~5.14]{CheCol97} allows one to show that there is an open dense subset on which $X$ is a topological manifold.  In the case where the $M_j$ are boundary free it was was shown in \cite[Th.~6.1]{CheCol97} that $X$ is a manifold away from a codimension 2 set.  More recently, it was shown by Bru\' e-Naber-Semola \cite{BruNabSem22} that even in the boundary case one has a manifold structure away from a codimension 2 set.  That is, the top stratum of the singular set $\cS^{n-1}(X)$ is itself a manifold away from a codim 2 set, and thus $X$ is a topological manifold with boundary away from a codim 2 set.\\

\subsection{Main Result on Topological Structure}

It has remained an open question as to whether in the collapsed case $X$ needs to have a topological manifold structure on some open dense subset.  The main result of this paper is to answer this question in the negative:

\begin{theorem}\label{t:main}
Let $(X^4,h)$ be a smooth complete manifold with $\Ric_X>\lambda$, where $\lambda\in \dR$.  Then for every $\epsilon>0$ there exists a metric space $(X^4_\epsilon,d_\epsilon)$ such that
\begin{enumerate}
	\item $d_{GH}\Big(X^4,X_\epsilon\Big)<\epsilon$ ,
	\item $X^4_\epsilon$ is $4$-rectifiable ,
	\item There exist $(M^6_j,g_j)\stackrel{GH}{\longrightarrow} (X^4_\epsilon,d)$ with $\Ric_{g_j}>\lambda$ ,
	\item $\forall$ open set $U\subseteq X_\epsilon$ we have that the homology group $H_2(U)$ is infinitely generated.  Consequently, every open set $U\subseteq X$ is noncontractible and therefore not homeomorphic to Euclidean space.
\end{enumerate}
\end{theorem}

	We will see that it is possible to build many such $X^4_\epsilon$.  In short, for each countable dense subset $\cC=\{x_j\}\subseteq X$ and each collection of sufficiently decaying constants $\epsilon\geq \epsilon_j\to 0$ we will build $X_\epsilon$ by connect-summing $X^4$ with a $\dC P^2$ of size $\approx\eps_j$ at the collection of points $\cC$.  The topological picture will be similar to a complex algebraic blow up, where we replace a point $x_j$ with a $2$-sphere $S^2_{\eps_j}$, though it is important to note that this is purely topological and there is no complex structure being preserved in this process.  In particular, the geometric properties of the {\it blow down} map that sends the newly added $S^2_{\eps_j}$ to the chosen $x_j$ will be important to the construction.  These blow down maps will explain how the rectifiable charts of the space $X^4_\epsilon$ collapse each of our added $2$-spheres to points, which form a set of measure zero.  Most examples of rectifiable structures which are not manifolds arise by allowing for holes in the space.  The blow down picture here explains the ability to build a rectifiable space which is nowhere a manifold, but also has no holes.  By choosing these points $\{x_j\}$ and scales $\epsilon_j$ fairly freely we see that a smooth structure is actually quite hard to obtain under higher dimensional lower Ricci curvature bounds, and in a certain generic sense we should expect limit spaces to not have manifold structures. \\

The above raises the question about whether if we assume bounds on topology we might obtain more.  There are two versions of such a question: 

\begin{question}
Let $(M^n_j,g_j,\nu_j,p_j)\stackrel{mGH}{\longrightarrow}(X^k,d,\nu,p)$ where $\Ric_j\geq \lambda$, $\nu_j := \frac{dv_{g_j}}{Vol(B_1(p_j))}$ and such that $|H_*(M_j,\dZ)|<A<\infty$.  Then is $X^k$ a topological manifold on an open dense subset?
\end{question}

Here, $|H_*(M_j,\dZ)|$ refers to the number of generators for the abelian group in question. We can phrase the question directly on the metric-measure space itself in a very similar form:

\begin{question}
	If $(X^k,d,\nu)$ is an $RCD(n,\lambda)$ space with bounded homology $|H_*(X,\dZ)|<\infty$, then is $X$ a topological manifold on an open dense subset?
\end{question}

Rephrasing the above: is the nonmanifold structure in Theorem \ref{t:main} only possible in the presence of infinitely generated homology?  

\vspace{.3cm}

\subsection*{Acknowledgments}

The second author was funded by National Science Foundation Grant No. DMS-1809011 for much of this work.

\vspace{.5cm}

\section{Geometric Outline of Construction}

Let us turn our attention to the construction of the smooth manifolds $(M^6_j,g_j) := (X^4_j\times S^2, h_j+f_j^2 g_{S^2})$.  We will often write $X_j\times_{f_j} S^2$ to represent that we are geometrically considering the product of two spaces with a warping factor $f_j$.  We can view $X^4_j$ as the blow up of $X^4$ at an increasingly dense sequence of points.  That is, to construct $X^4_j$ we will effectively take a collection of points $\{x^a_j\}\subseteq X^4$ and replace these points with $2$-spheres.  Each time we blow up a point, we are introducing a new noncontractible $S^2$ into the space.  Geometrically, each $S^2$ being introduced will be of size at most $\epsilon$, but their sizes will decrease quickly as the sequence continues.  We will additionally alter the warping factor on the $S^2$ factor of $M_j^6$ in order to preserve the strict Ricci curvature lower bound.  As our collection of blow up points becomes dense, we will arrive at our limit $X^4_\epsilon$.\\

Our construction will be inductive.  That is, given $M^6_j$ which satisfies a handful of inductive properties, we will explicitly construct from it $M^6_{j+1}$ with similar inductive properties.  Our goals in this section will first be to prove Theorem \ref{t:main} given our inductive sequence, which we will do in Section \ref{ss:main_thm_proof_induction}.  We will then focus on the inductive construction itself, which will be broken down into steps.  Each step will consist of a main Inductive Lemma.  The Inductive Lemmas will be proved later in the paper, but in the meantime we will finish the inductive construction in Section \ref{ss:geom_construction:step3} based on these Lemmas.  As such, our goal for this Section is to complete the proof of Theorem \ref{t:main} modulo the proof of the Inductive Lemmas.\\

In order to state the conditions of our inductive construction, let us introduce the correct notion of regularity scale for this paper. 

\begin{definition}[Regularity Scale]
	Let $(M^n,g) = (X^{n-2}\times S^2,h+f^2g_{S^2})$ be a smooth manifold with $x\in X$, and $0< \eta <1$ a constant.  Then we define the regularity scale
\begin{align}
	r_x=r^\eta_x := \max_{0<r\leq 1 }\Big\{\inj_X(x)\geq r \text{ and } \sup_{B_r(x)}\sum_{0\leq k\leq 2}\Big(r^{2+k}|\nabla^{k}\Rm_h|+r^{1+k}|\nabla^{1 + k}\ln f|\Big) \leq \eta\Big\}\, .
\end{align}
\end{definition}
\begin{remark}
	The definition depends on a constant $0<\eta<1$, though this constant may be fixed somewhat arbitrarily.  Pictorially, when $\eta$ is small we are increasingly close to $\dR^{n - 2} \times S^2$ with a product metric.
\end{remark}
\begin{remark}
	It follows that if $r_x>2r$ then topologically $B_r(x)\subseteq X$ is contained in a Euclidean ball.  
\end{remark}

\begin{remark}	
	If $r_x>2r$ then we can write the metric $h$ in exponential coordinates on $B_r(x)$ so that $h_{ab}$ is $C^{2}$ close to the identity and $\ln f$ is $C^{3}$ close to a constant.  See Lemma \ref{l:exponential_regularity}.
\end{remark}
\vspace{.2cm}

Now let $(X^4,h)$ be our chosen manifold from Theorem \ref{t:main} with $\Ric_X>\lambda$.  We will go ahead and assume $X^4$ is compact, but it will be clear that this is not needed as all constructions are purely local.  Primarily, this allows us to discuss the construction in terms of some global parameters instead of choosing them locally.  Let us define
\begin{align}\label{e:construction_outline:Ricci_lower}
	\lambda^+ := \max\big\{\lambda: \Ric_X[v,v]\geq \lambda|v|^2\big\}\, ,
\end{align}
and observe that as $X$ is compact we have $\lambda^+>\lambda$ .  There is no loss in generality in us then assuming that the constant $\epsilon>0$ from Theorem \ref{t:main} satisfies
\begin{align}
	\epsilon < \frac{1}{2}(\lambda^+-\lambda)\, .
\end{align}

Our inductive construction of $(M^6_j,g_j)$ will produce a choice of parameters
\begin{align}
	r_j := \frac{1}{2^j}\, ,\;\; \delta_j\leq \delta^{1+j}\, ,\;\; \epsilon_j := \frac{\epsilon}{2^j}<\epsilon\, ,\;\;\lambda_j := \lambda^+-\sum_{k =1}^j \epsilon_k>\lambda\, ,
\end{align}

where $0<\delta<<1$ will later be chosen sufficiently small. We will let our base step of the induction be represented by the space $M^6_0 := X^4\times S^2$ with product metric $g_0 = h + \delta^2 g_{S^2}$.  Our inductive assumptions will be the following:\\

\begin{enumerate}
\item[(I1)]  We can write $M^6_j=X^4_j\times S^2$ with metric $g_j = h_j+f_j^2 g_{S^2}$ satisfying $\Ric_{g_j}\geq \lambda_j$.  The space $(X^4_j,h_j)$ and warping function $f_j:X_j\to \dR^+$ are smooth with $|f_j|<\delta_j r_j$.
\item[(I2)] We have a smooth mapping $\phi_j: X_j \to X_{j - 1}$ and a maximal disjoint collection $\{B_{r_j}(x^a_{j})\subseteq X_j\}_a$ subject to the conditions:
\subitem(I2.a) $\{x_j^a\}_a \subseteq\{x\in X_j : r_x\geq 4r_{j}\}$ , 
\subitem(I2.b) $\phi_{ji}\big(B_{4r_j}(x^a_j)\big)\cap \{x^b_i\}=\emptyset$, where $\phi_{ji} :=\phi_{i + 1}\circ\cdots\circ \phi_{j}:X_j\to X_i$ for $i < j$ . 
\item[(I3)] The mapping $\phi_{j}:X_{j}\to X_{j - 1}$ is a diffeomorphism away from $\phi_{j}^{-1}(x^a_{j - 1})\cong S^2$, an isometry away from $\phi_{j}^{-1}(B_{r_{j -  1}}(x^a_{j-1}))$, and $\phi_{j}^{-1}(B_{r_{j -  1}}(x^a_{j-1}))$ is diffeomorphic to the generating line bundle $E\to S^2$ .  
\item[(I4)] $\phi_{j}^{-1}(x^a_{j-1})\cong S^2$ are totally geodesic, round $2$-spheres of radius $\leq \delta_{j}r_j$ in $X_{j}$ .  We have $|D \phi_{j}|<C(6)$,  and if $x,y\in X_{j}$ with $d(x,y)>\delta_{j}r_j$ then $(1-\delta_{j})d(x,y)<d(\phi_{j}(x),\phi_{j}(y))<(1+\delta_{j})d(x,y)$.
\end{enumerate}

\vspace{.3cm}

The notation $C(6)$ above tells us that $|D\phi|$ is uniformly bounded by a dimensional $n=6$ constant, which in this case is just a uniform constant. \\

Let us discuss some of the above properties and their implications. Condition (I1) tells us that from a geometric standpoint $M_j^6$ is globally a warped product over $X^4_j$, and that geometrically the $S^2$ factor is disappearing in the limit.  Condition (I2) is an enumeration of the points we will be doing surgery around to move from $X_j$ to $X_{j+1}$ .  The important point to observe is that necessarily this set is becoming increasingly dense by Condition (I2.a), as the points are maximal subsets inside the set whose regularity scale is too large.\\

To move from $X_j$ to $X_{j+1}$ we will be performing surgeries on the balls $B_{r_j}(x^a_j)$.  Condition (I3) is telling us that the surgery is topologically a connect sum with $\dC P^2$, where we are replacing each $x^a_j$ with a $2$-sphere.  Near $x^a_j$ this has the effect of replacing the diffeomorphic ball $B_{r_j}(x^a_j)$ with the total space $E\to S^2$ of the generating line bundle over $S^2$.  Note that the unit sphere bundle in $E$ is $S^3$, and hence this is the right object for gluing.  From a topological perspective, moving from $X_j$ to $X_{j+1}$ adds a second homology generator for each $x^a_j$.  Condition (I2.b) is telling us that our surgeries never intersect the previously added $2$-spheres.\\

Condition (I4) is explaining the geometric properties of our blow down maps $\phi_j:X_j\to X_{j-1}$.  These will each be smooth mappings, indeed uniformly lipschitz, and will be Riemannian isometries away from some small neighborhoods of the blow up points.  We will see that the limit map $\Phi:= \lim_{j \to \infty}\phi_{j1}:X^4_\epsilon\to X^4$ is uniformly H\"older and locally bi-Lipschitz away from a set of measure zero.  In particular we will have a single rectifiable chart of $X^4_\epsilon$ over $X^4$, that is an a.e.\@ defined locally bi-Lipschitz map onto a full-measure subset of a smooth manifold.  The blow up 2-spheres will all be collapsed to single points under this mapping.

\subsubsection{Proving Theorem \ref{t:main} given the Inductive Spaces $M^6_j$}\label{ss:main_thm_proof_induction}

Before focusing on the inductive construction itself, let us see how to use (I1)-(I4) in order to finish the proof of Theorem \ref{t:main}.\\

Let us begin by studying properties of the spaces $X^4_j$.  For each $i<j$ let us denote $\{S^2_{jia}\}_{a} := \{\phi_{ji}^{-1}(x^a_i)\}_{a}$.  Notice by conditions (I4) and (I2.b) that these are disjoint totally geodesic round $2$-spheres inside of $X^4_j$.  Additionally, by (I3) and a Mayer-Vietoris sequence we have that
\begin{align}
	\pi_1(X^4_j) = \pi_1(X^4) \text{ with }H_2(X^4_j)= H_2(X^4)\oplus\bigoplus_{i < j\,,\, a} \langle[S^2_{jia}] \rangle\, .
\end{align}
In particular, the rank of the second homology is growing in step with these two spheres.\\

Let us now flesh out the geometric properties implied by (I4) more completely.  From (I4), we expect the lipschitz constant to be uniformly bounded, but not necessarily close to 1.  On the other hand, (I4) also tells us that the lipschitz constant is close to 1 away from a small neighborhood of the diagonal in $X^4_j \times X^4_j$.  Consequently, we have the bounds
\begin{align}\label{e:phi_lipschitz_properties}
	&\text{If $x,y\in X_{j}$ with $d(x,y)>\delta_{j}r_j$ then $(1-\delta_{j})d(x,y)<d(\phi_{j}(x),\phi_{j}(y))<(1+\delta_{j})d(x,y)$}\, ,\notag\\
	&\text{If $x,y\in X_{j}$ with $d(x,y)\leq \delta_{j}r_j$ then $d(\phi_{j}(x),\phi_{j}(y))<C d(x,y)$ for $C=C(6)$ .}
\end{align}

Thus we can take the lipschitz constant to be small if the distance between two points is also not too small.  This in particular implies the much weaker estimate 
\begin{align}\label{e:phi_GH_map}
	\text{ $\phi_j$ is a $C\delta_j \leq C2^{-j}\delta $-Gromov Hausdorff map}\, .
\end{align}
By composing, we see that $X^4_j$ is always $C\delta<\epsilon$-GH close to $X^4$ for small enough $\delta$ .

The importance of these estimates is that our real goal is to obtain uniform continuity for the maps $\phi_{ji}:X_j\to X_i$ from (I2.b).  It is too much to ask that they be uniformly lipschitz.  However for $\delta<\delta(C)=\delta(6)$ in the construction, condition (I4) will now allow us to show that the $\phi_{ji}$ are uniformly $C^\alpha$ maps.  Indeed, for any fixed $x,y\in X_j$ write $r :=d(x,y)$, and then using \eqref{e:phi_lipschitz_properties} we can estimate
\begin{align}
	d\big(\phi_{ji}(x),\phi_{ji}(y)\big) &\leq \prod_{k \leq j \,:\, r\leq \delta_kr_k}C\cdot \prod_{k \leq j\,:\, r>\delta_kr_k}(1+\delta_k)r\, ,
\end{align}

To estimate the above we will use that $\delta_j\leq \delta^{1+j}$, which gives us that
\begin{align}\label{e:phi_Holder}
	d\big(\phi_{ji}(x),\phi_{ji}(y)\big) &\leq (1+\delta)C^{\ln (r)/ \ln(\delta/2)} r = (1+\delta) r^{\alpha(\delta)}\, ,
\end{align}

where $\alpha(\delta)=1+\ln (C)/\ln(\delta/2)\to 1$ as $\delta\to 0$ . \\

Now recall by \eqref{e:phi_GH_map} that $\phi_j$ is a $C\delta_j\leq C2^{-j}\delta$-Gromov Hausdorff map.  Consequently we have that $\phi_{ji}:X_j\to X_i$ is a $C\sum_{j\geq k\geq i+1}\delta_k \leq C 2^{-i}\delta$ Gromov Hausdorff map.  This tells us that $\{X_j\}$ is a Cauchy sequence and so
\begin{align}
	(X_j,d_j)\stackrel{GH}{\longrightarrow} (X_\epsilon,d)\, .
\end{align}
Note that this is not a subsequential convergence but an actual convergence by the Cauchy condition.  It follows from (I1) that $M_j\stackrel{GH}{\longrightarrow} X^4_\epsilon$ as well since $|f_j|\leq \delta_j r_j\to 0$.  As the maps $\phi_{ji}$ witness the GH-Cauchy condition, we can take limits
\begin{align}
	\lim_{j\to\infty}\phi_{ji}:= \Phi_i:X^4_\epsilon\to X_i\, ,
\end{align}
where the $\Phi_i$ are also $C2^{-i}\delta$ Gromov Hausdorff maps.  It follows from \eqref{e:phi_Holder} that the $\Phi_i$ are $C^{\alpha}$-H\"older maps\footnote{In fact, the $\Phi_i$ are rectifiable charts. Indeed, the $\phi_j$ are $(1 + \delta_j)$-bi-Lipschitz away from the bubbles $\phi_j^{-1}(B_{\delta_{j - 1}r_{j-1}}(x_{j-1}^a))$, and uniformly locally bi-Lipschitz away from the added $2$-spheres. Composing these estimates as in \eqref{e:phi_Holder} shows that for any $j > 0$, the $\Phi_i$ are $C(i,j)$-bi-Lipschitz away from $\bigcup_{k > i,a}B_{r_k\delta_{j\wedge k}}(S^2_{ka})$.}, or more precisely that
\begin{align} \label{e:phi_Holder_conclusion}
	d\big(\Phi_i(x),\Phi_i(y)\big) \leq (1+\delta) d(x,y)^\alpha\, .
\end{align}
Importantly, we have that the $\Phi_i$ are continuous maps.\\

Now recall from (I3) and (I4) that $\{S^2_{jia}\}\subseteq X^4_j$ are totally geodesic round $2$-spheres in $X^4_j$.  Note also that for $i<j<k$ we have by (I2.b) and (I3) that $\left.\phi_{kj}\right|_{S^2_{kia}}:S^2_{kia}\to S^2_{jia}$ is an isometry.  Consequently, we can limit our sequences of $2$-spheres to get round $2$-spheres
\begin{align}
	\lim_{j\to\infty}\{S^2_{jia}\subseteq X_j\} = \{S^2_{ia}\subseteq X^4_\epsilon\}\, .
\end{align}
Note that $\left.\Phi_j\right|_{S^2_{ia}}:S^2_{ia}\to S^2_{jia}$ is an isometry for every $j>i$ , and in particular the radius of each $2$-sphere $S^2_{ia}$ is at most $r_i\delta_i$.  Now we claim that each $S^2_{ia}\subseteq X^4_\epsilon$ is a nontrivial generator in the second homology group as well.  Indeed, assuming this is not the case, there must be a continuous $3$-chain $\psi:\Delta^3\to X^4_\epsilon$ with boundary $\partial\psi = S^2_{ia}$.  If we compose with $\Phi_j$ then this gives us a continuous chain $\Phi_j\circ \psi:\Delta^3\to X^4_j$ whose boundary is the $2$-sphere $S^2_{jia}$.  However, as we know that $S^2_{jia}$ is a nontrivial generator in the second homology group, this is not possible.  A similar argument shows that each $\{S^2_{ia}\}$ generates an independent factor in the second homology group. \\

Finally, let us show more carefully that the $2$-spheres $\{S^2_{ia}\}_{i,a}$ are dense in $X_\epsilon$. Fix any $x\in X_\epsilon$ and $\epsilon'>0$. Consider $\Phi_i(x)\in X_i$ for some large $i$ such that $\Phi_i$ is an $\epsilon'$-GH map. We \textit{claim} that we can find $x^a_j\in X_j$ for some $j\geq i$ such that $d(\phi_{ji}(x^a_j), \Phi_i(x))<\epsilon'$. Since $\Phi_i=\phi_{ji}\circ \Phi_j$, it is then clear that the $2$-sphere $S^2_{ja}=\Phi^{-1}_j(x_j^a)$ is contained in the ball $B_{2\epsilon'}(x)$. \textit{Proof of claim:} Suppose it is not true. Since $\phi_{ji}$ is a $C2^{-i}\delta$-Gromov Hausdorff map, we then have $x^a_j\not\in B_{\epsilon'-C2^{-i}\delta}(\Phi_j(x))$ for any $j\geq i$ and any blow-up point $x^a_j$. Thus the ball $B_{\epsilon'-C2^{-i}\delta-r_i}(\Phi_j(x))$ has the same Riemannian metric for all such $j$. In particular, $\Phi_j(x)$ will have the same regularity scale for all such $j$ (note that the regularity scale is invariant under scaling of the warping function $f$). Now for $j$ large enough, a point in $B_{\epsilon'-C2^{-i}\delta-r_i}(\Phi_j(x))$ has to be blown-up since the collection of blow-up points is maximal. This is a contradiction. \\
%\begin{align}
%d(x, y)&\leq d(\Phi_j(x), \Phi_j(y))+\epsilon_1\notag\\
%&=d(\Phi_j(x), \phi_{kj}(x^a_k))+\epsilon_1 \quad \text{(since $\phi_{kj}\circ \Phi_k=\Phi_j$)}\notag\\
%&<\epsilon_2+\epsilon_1.
%\end{align}

  Thus we have our limit space $X^4_\eps$ and a dense collection of two spheres $\{S^2_{ia}\}$ which are all generators in the second homology group, as claimed.  This finishes the proof of Theorem \ref{t:main} under the assumption that we have built our inductive sequence $M_j^6$ satisfying (I1)-(I4).   $\qed$

\vspace{.3cm}

\subsection{Step 1:  The Gluing Block $\cB(\epsilon,\alpha,\delta)$}

In order to prove Theorem \ref{t:main} we are therefore left with showing how to build $M_{j+1}=X_{j+1}\times_{f_{j + 1}} S^2$ from $M_j=X_j\times_{f_j} S^2$ in the inductive construction.  The first step of the construction will build what is our main gluing block.  When we move from $X_j$ to $X_{j+1}$ we will take our appropriately dense collection of points $\{x^a_j\}$ and replace a small neighborhood of each with our gluing block $\cB$.  From a topological perspective, it will be a connect sum with a copy of $\dC P^2$ near each $x^a_j$, so that we are blowing up the points $\{x^a_j\}$ and replacing them with $2$-spheres.\\

Our main constructive Lemma in this step of the construction is the following:\\

\begin{lemma}[Inductive Step 1]\label{l:inductive_step1}
	For every $0<\epsilon<\frac{1}{10}$, $0<\alpha<\alpha(\epsilon)$ and $0<\delta<\delta(\epsilon,\alpha)$ there exists a smooth Riemannian manifold $\cB(\epsilon,\alpha,\delta)$ with   $\Ric_{\cB}>0$ and such that
\begin{enumerate}
	\item $\cB$ is diffeomorphic to $E\times S^2$, where $E$ is the total space of the generating line bundle $E\to S^2$.  Further, $(\cB,g)$ has a warped product structure $g:= g_E+f^2 g_{S^2}$.
	\item There exists $U=U_E\times S^2\subseteq \cB$ such that $\cB\setminus U$ is isometric to a neighborhood of infinity in $C(S^3_{1-\epsilon})\times_{\delta r^\alpha} S^2$ with the metric $dr^2+(1-\epsilon)^2 r^2 g_{S^3}+\delta^2 r^{2\alpha} g_{S^2}$ .
\end{enumerate} 
Further, there exists $\phi:(E,g_E)\to C(S^3_{1-\epsilon})$ such that
\begin{enumerate}
	\item $\phi$ is a diffeomorphism away from the cone point $0\in C(S^3)$, with $\phi^{-1}(0)\cong S^2$ an isometric sphere, 
	\item $|D \phi|\leq C$ is uniformly bounded with $\phi$ an isometry away from $U_E$ .
\end{enumerate}
\end{lemma}

Note that $\cB\setminus U$ looks very pGH-close to $\dR^4\times S^2$, where the degree of closeness is being measured by $\epsilon,\alpha, \delta$ in a quantitative manner.

\vspace{.3cm}

\subsection{Step 2:  Adding Cone Singularities}

The second step of the construction involves changing the geometry near each $x^a_j$ so that the gluing blocks $\cB$, which have a very specific geometry at infinity, may be isometrically glued along an annulus into $X_j$.  Let us begin with a broader discussion before stating the main constructive Lemma.\\

Let us start with a discussion of the density of singularities.  It is known, see for instance \cite[Ex.~2]{OtsShi94}, that one can build examples $(X,h)$ with $\Ric_h>0$ for which the singular set is dense.  In fact, the example in \cite{OtsShi94} has positive sectional $\text{sec}_h>0$ and even full positive curvature operator $\Rm_h>0$.  What is at first counter-intuitive is that the constructions of \cite{OtsShi94} not only produce but essentially rely on these stronger curvature conditions.  That is, the construction of singularities with $\Rm_h>0$ in \cite{OtsShi94} is very analogous to the construction of convex functions with nonsmooth points.\\

Let us now consider what is almost the reverse direction.  Begin with a smooth space $(X,h)$ with some form of lower bound on the curvature and ask about adding cone singularities near any point $x\in X$ without destroying the lower curvature bound.  If for instance $\text{sec}_h> \lambda>0$ then one can accomplish this by performing a $C^0$ gluing in the spirit of \cite[Sec.~4]{Per97}.  Namely, one can remove a sufficiently small neighborhood of $x\in X$ and isometrically glue in a rescaled spherical suspension of the boundary.  There will be what is essentially distributional curvature added along the gluing, but the $\text{sec}_h> \lambda$ assumption will allow us to guarantee that these distributional curvatures all have the right sign.  In particular, one can smooth near the gluing region and preserve the $\text{sec}_h>\lambda>0$ condition. \\ 

The procedure described above of adding cone singularities near any point does not work if we are only assuming $\Ric_h>\lambda$.  In short, the distributional curvature added from the $C^0$ gluing is due to the difference in second fundamental forms of the boundary on the two sides of the gluing.  This second fundamental form in turn is closely related to sectional curvature, and Ricci curvature control is not sufficient.  It turns out that we need to exploit better the local geometry near $x\in X$ in order to control the Ricci curvature.  \\

In Step 2.1 of Section \ref{s:step2} we will see how to resolve this problem and add such cone singularities to arbitrary spaces satisfying $\Ric_h>\lambda$ without destroying the lower Ricci curvature bound\footnote{The construction in Section \ref{ss:step2:1} is a bit more general; one can drop the warping factor to obtain the claimed result.}.  The inductive Lemma of this Step of the construction is a generalization of this discussion and will allow us to also add conical singularities with a fixed warping structure near every point.  This extra warping control is necessary in order to use the inductive Lemma of Step 1 to add our desired topology.  We will focus on the $6$-dimensional case $M=X^4\times S^2$ of interest, though it is clear dimension is not a relevant constraint.  Precisely, we have the following local gluing Lemma, which focuses on a ball with controlled regularity scale:\\

\begin{lemma}[Inductive Step 2]\label{l:inductive_step2}
	Consider a warped product space $(B_2^4(p) \times S^2,g)$ with metric $g = g_B + f^2g_{S^2}$ satisfying $\inj_{g_B}(p) > 2$, $\Ric_g>\lambda$, and
\begin{align}
	|\Rm_{g_B}|\, ,\, |\nabla\Rm_{g_B}|\, ,\, |\nabla^2\Rm_{g_B}|\, ,\, |\nabla\ln f|\, ,\, |\nabla^2\ln f|\, ,\, |\nabla^3 \ln f| < \eta<1\, .
\end{align}
Write $r:= \dist_{g_B}(\cdot, p)$. Then for all choices of parameters $0 < \eps < \eps(|\lambda|)$, $0<\alpha < \alpha(\eps)$, $0 < \hat{r} < \hat r(\alpha, \lambda,\eps)$, and $0 < \hat{\delta} < \hat{\delta}(\lambda,\alpha,\epsilon,||f||_{L^\infty},\hat{r})$, there exists $0 < \delta = \delta(\,\hat{\delta}\|f\|_\infty\mid \alpha, \lambda, \eps)$ and a warped product metric $\hat g = \hat{g}_B+\hat f^{\,2} g_{S^2}$ such that:
\begin{enumerate}
	\item The Ricci lower bound $\Ric_{\hat g}>\lambda - C(4)\eps$ holds for $\hat r/2 \leq r \leq 2$\, ,
	\item $\hat g = g_B+\hat{\delta}^2 f^2 g_{S^2}$ is unchanged up to scaling the warping factor $f$ by $\hat \delta$ for $1 \leq r \leq 2$ ,
	\item $\hat g = dr^2+(1-\eps)^2 r^2 g_{S^3}+\delta^2 r^{2\alpha} g_{S^2}$ has the cone warping structure $C(S^3_{1 - \eps}) \times_{\delta r^\alpha}S^2$ for $r\leq \hat r$ ,
	\item The identity map $\Id:(B_2(p),\hat{g}_B)\to (B_2(p),g_B)$ is $(1+2\epsilon)$-bi-Lipschitz.
\end{enumerate}
\end{lemma}
\begin{remark}
	Our notation of the constant dependence $ \delta(\,\hat{\delta}\|f\|_\infty\mid\alpha, \lambda, \eps)$ means that $\delta\to 0$ as $\hat{\delta}\|f\|_\infty\to 0$ with the other constants $\alpha,\lambda,\epsilon$ fixed.
\end{remark}

\begin{remark}
The caveat in $(1)$ that $\Ric_{\hat{g}} > \lambda - C(4)\eps$ only away from a small neighborhood of the cone point, e.g. $B_{\hat{r}/2}(p) \times S^2$ , cannot be improved. This is simply due to the fact that the fixed warping structure $C(S^3_{1 - \eps}) \times_{\delta r^\alpha} S^2$ has infinite negative curvature at the pole, in particular it has $\left.\Ric\right|_{TS^2}\to -\infty$ as $r \to 0$. In practice, $B_{\hat{r}/2}(p) \times S^2$ will be replaced by a rescaled bubble $\cB$ obtained from Lemma \ref{l:inductive_step1}, which does have the appropriate Ricci lower bound.
\end{remark}

In practice the above works as follows.  If $M=X\times_f S^2$ is a smooth manifold with a lower Ricci curvature bound and $x\in X$, then the above tells us we can find a potentially very small neighborhood $B_{2\rho}(x)\times S^2$ in which we can change the geometry of $M$.  Specifically, after shrinking the warping factor by $\hat{\delta}$, we can alter the metric so that the ball $B_{\hat{r}\rho}(x)\times S^2$ will be isometric to that of the warped cone $C(S^3_{1-\epsilon})\times_{\delta \rho(r/\rho)^\alpha} S^2$.\\  

In the case of no warping factor, as per the discussion before the statement of the above Lemma, we can repeat this process indefinitely in order to produce a dense set of singularities.  In the case of a warping factor we can similarly repeat this process indefinitely, but also combine with the Inductive Lemma \ref{l:inductive_step1} in order to glue topology in at each step.  We will discuss this construction more carefully in the next step.

\vspace{.3cm}

\subsection{Step 3:  Constructing $M_{j+1}$}\label{ss:geom_construction:step3}

Let us now see how to use the Inductive Lemmas \ref{l:inductive_step1} and \ref{l:inductive_step2} in order to complete the inductive step of the construction and build $M_{j+1}$ from $M_j$. Thus let us assume we have built $M_j=X_j\times S^2$ with $g_j= h_j +f_j^2 g_{S^2}$ and $\Ric_{g_j}>\lambda_j$.  Let us also choose a maximal disjoint collection $\{B_{r_j}(x^a_j)\}$ such that $\{x_j^a\} \subseteq\{ x\in X^4_j: r_x\geq 4 r_j\}$ and $\phi_{ji}(B_{4r_j}(x_j^a)) \cap \{x_i^b\} = \emptyset$. \\ 

Due to the unfortunate number of constants being accounted for in the construction, it is helpful to briefly remark on what will happen.  The goal is to construct $M_{j+1}$ in two steps.  First we will take each ball $B_{r_j}(x^a_j)\subseteq X_j$ and apply the Inductive Lemma \ref{l:inductive_step2} in order to add a warped cone singularity on $B_{\hat r r_j}(x^a_j)$.  This space is not smooth at the cone point $x^a_j$ of course, but will have appropriately positive Ricci at least outside of $B_{\hat{r}r_j/2}(x^a_j)$.  We will then apply the Inductive Lemma \ref{l:inductive_step1} in order to replace each warped cone $B_{\hat r r_j}(x^a_j)$ with the smooth bubble metric $\hat\cB$ of positive Ricci.  This will produce $X_{j+1}$, and if we choose the various constants sufficiently small at each stage, we can do this while keeping control of both the space and its relationship to $X_j$.  That is, we can show the inductive hypotheses $(I1)-(I4)$ are satisfied.\\

Let us now describe the modifications to each disjoint ball $B_{2r_j}(x_j^a) \times S^2$ in more detail. It will in fact be more convenient to describe the process on a single rescaled ball $B_2(p) \times S^2 := r_j^{-1}(B_{2r_j}(x_j^a) \times S^2)$ with metric $g = g_B + f^2g_{S^2} :=r_j^{-2}(h_j + f_j^2g_{S^2}) = r_j^{-2}g_j$. Observe that the ball $B_2(p) \times S^2$ satisfies the criteria of Lemma \ref{l:inductive_step2}, with Ricci lower bound $r_j^2\lambda_j$.  Therefore let us apply Lemma \ref{l:inductive_step2} with input parameters $\eps'$, $\alpha$, $\hat{r}$, $\hat{\delta}$, and output parameter $\delta_I$.  We choose $\eps'$, $\alpha$, $\hat{r}$, and $\hat{\delta}$ to satisfy the conditions of Lemma \ref{l:inductive_step2} and Lemma \ref{l:inductive_step1}, but we may further shrink these constants later.\\

The previous application of Lemma \ref{l:inductive_step2} to $B_2(p) \times S^2$ produced a metric $\hat g = \hat g_{B}+\hat f^2 g_{S^2}$ such that $\hat f= \hat{\delta}f$ and $\hat g_B = g_B$ on the open set $A_{1,2}(p)$. Note additionally that by choosing $\eps' < \eps'(\eps_{j + 1}, r_j)$ small enough, we can ensure that $\Ric_{\hat{g}} \geq r_j^2\lambda_{j + 1}$ holds away from $B_{\hat{r}/2}(p) \times S^2$. Inside the inner radius, $(B_{\hat r}(p) \times S^2,\hat g)$ is isometric to the cone warping structure $B_{\hat r}(0) \times S^2 \subseteq C(S^3_{1-\epsilon'})\times_{\delta_I r^\alpha} S^2$.  We wish to replace this inner ball with a rescaled copy of the gluing model $\cB(\epsilon',\alpha,\delta(\epsilon',\alpha))$ from Lemma \ref{l:inductive_step1}.  Recall from Lemma \ref{l:inductive_step1} that $U\subseteq \cB$ is an open set  so that $\cB\setminus U$ is isometric to $dr^2 + (1-\epsilon')^2 r^2 g_{S^3}+(\delta(\eps',\alpha))^2r^{2\alpha} g_{S^2}$.  Let $\hat \cB$ be a rescaling of $\cB$ so that the rescaled core $\hat{U} \subseteq \hat \cB$ is contained in a bounded region with a collar neighborhood $\hat V$ of its boundary isometric to $A_{\hat{r}/2,\hat{r}}(0) \times S^2 \subseteq C(S^3_{1 - \eps'}) \times_{\delta_{II} r^\alpha}S^2$, for some $\delta_{II}=\delta_{II}(\hat{r}, \alpha, \eps')$. We additionally require that $\hat{r} < \hat{r}(|\lambda_{j + 1}|, r_j)$ is small enough so that $\Ric_{\hat\cB} > r_j^2\lambda_{j + 1}$ in this region.  In particular, if we multiply $\hat f $ by $\min(\delta_{I},\delta_{II}) /\delta_I$ and multiply the warping factor of $\hat \cB$ by $\min(\delta_{I},\delta_{II}) /\delta_{II}$, then we can define a warped product $(\bar \cB, \bar g)$ from $(B_2(p) \times S^2,\hat g)$ by removing $B_{\hat r/2}(p)\times S^2$ and isometrically gluing the rescaled model $\hat\cB$ into $B_{\hat{r}}(p)\times S^2$ along $\hat V$. Note that for warped product geometries, multiplying the warping factor by a constant $\leq 1$ only increases the Ricci curvature--see Remark \ref{rmk:mult-by-delta}--so $\Ric_{\bar g} > r_j^2\lambda_{j + 1}$. \\

We now seek to replace the original ball $B_{2r_j}(x_j^a) \times S^2 \subseteq M_j$ with the rescaled $r_j\bar \cB$. By construction, $r_j\bar \cB$ is a warped product over a compact manifold with boundary, with a collar neighborhood of its boundary isometric to $(A_{r_j,2r_j}(p) \times S^2,h_j + (\bar{\delta}f_j)^2g_{S^2})$, where $\bar \delta := \hat{\delta}\min(\delta_I,\delta_{II})/\delta_I$. If we multiply the warping factor of $M_j$ by $\bar{\delta}$ (which is independent of the choice of point in $\{x_j^a\}$!), then we see that $B_{r_j}(x_j^a) \times S^2$ can be replaced by $r_j\bar \cB$, glued isometrically along $A_{r_j,2r_j}(p) \times S^2$. We similarly replace $B_{2r_j}(x_j^a) \times S^2$ for each other $a$ to form $(M_{j + 1}, g_{j + 1})$. \\

We have almost completed our construction.  For any choices of $\eps'>0$ and $\hat r>0$ appropriately small our construction above holds, and we need only make choices.   First observe that we can define $\phi_{j+1}:X_{j+1}\to X_j$ to be the identity outside of $\bigcup_a B_{r_j}(x^a_j)$.  On the union of annuli $\bigcup_a (B_{r_j}(x^a_j)\setminus B_{\hat{r} r_j}(x^a_j))$ we have by Lemma \ref{l:inductive_step2} that $|D\phi_{j+1}-I|<C\eps'$ if we set $\phi_{j + 1} := \Id$ setwise on this region.  We can use the second part of Lemma \ref{l:inductive_step1} in order to extend $\phi_{j+1}$ to each glued bubble $r_j\bar\cB$, so that $|D \phi_{j+1}|<C(6)$ on $\{B_{\hat{r} r_j}(x^a_j)\}$.  If we now choose $\epsilon'<\epsilon'(\delta_{j+1})$ and $\hat{r} < \delta_{j + 1}/2$ sufficiently small, then we have that \eqref{e:phi_lipschitz_properties} holds.  This finishes the construction of $M_{j+1}$.  $\qed$

\vspace{.5cm}

\section{Preliminaries}

\vspace{.3cm}

\subsection{Ricci Curvature of Warping Geometry}

The underlying ansatz for all constructions going forward is the warped product with $S^2$.  We have several formulas for the curvature of such spaces which will be used in this paper.  In this section we collect together some elementary remarks and formulas about such constructions.  These will be the starting point for many of the other formulas computed in this paper. \\

Let us now be more precise: consider the data of a smooth Riemannian manifold $(X, h)$ with a positive function $f: X \to (0,\infty)$. We can form the warped product of $X$ with $S^2$ as follows:

\begin{equation}
X \times_f S^2 := (X\times S^2, h + f^2g_{S^2})\,.
\end{equation}

That is, $X \times_f S^2$ is a Riemannian manifold that topologically has the structure of a (trivial) sphere bundle $S^2 \to X \times_f S^2 \stackrel{\pi}{\to} X$. The fibers $\pi^{-1}(x) = S^2_{f(x)}$, $x \in X$, of the projection map $\pi$ are metrically round spheres of radius $f$, and are orthogonal to the natural sections $X \times \{\omega\}$, $\omega\in S^2$. Denoting $M := X \times_f S^2$, we will therefore use the orthogonal splitting to make the identification $TM \cong TX \oplus TS^2$. We obtain the following concise formulas for the Ricci curvature of $M$ in the complementary directions:

\begin{align}
\left.\Ric_M\right|_{TX} &= \Ric_{h} - 2\frac{(\nabla_{h}^2 f)}{f}\,, \\
\left.\Ric_M\right|_{TS^2} &= \left(\frac{1}{f^2} - \frac{|\nabla_{h} f|^2_{h}}{f^2} - \frac{\Delta_{h}f}{f}\right)f^2g_{S^2}\,.
\end{align}

Let us make a couple of remarks about the general form of these identities, which we will use without further comment throughout the rest of the paper:

\begin{remark}
There is no cross-term in $\Ric_M$, i.e. $\Ric_M(v,w) = 0$ for $v \in TX$ and $w\in TS^2$. In practice, this splitting of the Ricci curvature into orthogonal blocks will allow us to subdivide the problem of lower bounding $\Ric_M$ into two distinct steps. 
\end{remark}

\begin{remark}
\label{rmk:mult-by-delta}
The scaling action $f \mapsto \lambda f$ for $\lambda > 0$ leaves the $TX$ directions $\left.\Ric_M\right|_{TX}$ invariant, and acts on the $TS^2$ directions by $\left.\Ric_M\right|_{TS^2} \mapsto \lambda^2\left.\Ric_M\right|_{TS^2} + (\lambda^{-2} - 1)\lambda^2 g_{S^2}$. Thus, $\Ric_M$ is non-decreasing under the scaling $f \mapsto \lambda f$ when $\lambda \leq 1$. In practice, this will mean that $\left.\Ric_M\right|_{TS^2}$ can be made as large as desired, without disrupting $\left.\Ric_M\right|_{TM}$, by multiplying $f$ by a suitably small positive number.
\end{remark}

\vspace{.3cm}

\subsection{$C^1$ Gluing Lemma for Warping Geometry}

We state in this subsection a $C^1$ gluing lemma.  It is a slight generalization of a result of Menguy \cite[Lem.~1.170]{Men00}, and its proof is essentially verbatim.  We will make use of it in both steps of the construction:

\begin{lemma}[$C^1$ Warped Geometry Gluing Lemma]\label{l:C1_warped_gluing}
	Let $(M^n, h)$ be a smooth manifold with $N^{n-1}\subseteq M$ a complete smooth submanifold.  Consider $(M^n\times S^k,g)$ where $g= h + f^2g_{S^k}$ is a $C^1$ warped product metric.  Assume further that
\begin{enumerate}
	\item  $h,f$ are smooth on $M\setminus N$ and have smooth limits from each side of $N$ ,
	\item $\inf\Ric_g > \lambda $ on $\big(M\setminus N\big) \times S^k$ . 
\end{enumerate}
Then for every open set $N\subseteq U$ and number $\epsilon > 0$ there exists a smooth metric $g_U= h_U+f^2_U g_{S^k}$ on $M\times S^k$ such that $\Ric_{g_U} > \lambda$ and $g_U = g$ outside of $U\times S^k$. Moreover, one can arrange that $\|g - g_U\|_{C^1(U\times S^k)} < \epsilon$ .
\end{lemma}
\begin{remark}
	It is enough to assume $g$ is $C^4$ on $M\setminus N$ .
\end{remark}
\begin{remark}
	The verbatim result is true for more general warping factors other than spheres.
\end{remark}

\vspace{.3cm}

\subsection{Regularity under the Exponential Map}

The following is relatively standard, and the proof goes through a series of Jacobi field estimates; however it is surprisingly difficult to find a precise reference for it.  For the convenience of the reader we state the result precisely below:

\begin{lemma}[Regularity of Exponential map]\label{l:exponential_regularity}
	Let $0 < \eta < 1$ be a number and $(B_1(p),g)$ a metric ball which satisfies the regularity scale estimates
\begin{align}\label{e:prelim:exponential:regularity_scale}
	\inj(p)>1\text{ with } \sum_{\ell = 0}^{k} ||\nabla^{\ell}Rm||_{L^\infty} < \eta\, .
\end{align}
Given an orthonormal basis $\{\partial_a\}$ of $T_pM$ let $g_{ab} = \exp^*g$ be the metric in exponential coordinates.  Then we can estimate
\begin{align}\label{e:exponential_regularity:estimate}
	\sum_{a,b}\left(r^{-2}||g_{ab}-\delta_{ab}||_{L^\infty} + \sum_{c}r^{-1}||\partial_{c} g_{ab}||_{L^\infty}+\sum_{\ell = 2}^k \sum_{c_1,\ldots, c_\ell}||\partial^{\ell}_{c_1,\ldots, c_\ell}g_{ab}||_{L^\infty}\right) \leq C(n,k)\,\eta\, .
\end{align}
\end{lemma}

\begin{remark}
Let us say a few words about how these estimates can be proved.  One rewrites the metric derivatives in terms of the Jacobi vector fields $J_a := r\partial_a$ along radial geodesics passing through p, e.g. $\partial_c g_{ab} = r^{-3} (g(\nabla_{J_c} J_a, J_b) + g(J_a,\nabla_{J_c}J_b))$ for $\partial_c \perp \nabla r$. To obtain estimates on the iterated covariant derivatives $J_{a_1,\ldots, a_k, b} := \nabla_{J_{a_1}}\cdots \nabla_{J_{a_k}} J_b$, one uses the equation they solve: $0 = \nabla_{J_{a_1}}\cdots \nabla_{J_{a_k}}\Big(J_b'' + R(J_b, \partial_r)\partial_r\Big) = J_{a_1,\ldots, a_k, b}'' + R(J_{a_1,\ldots, a_k, b}, \partial_r)\partial_r + E_{a_1,\ldots, a_k, b}$. The inhomogeneous term $E_{a_1,\ldots, a_k, b}$ depends only on lower-order covariant derivatives $J_{c_1,\ldots, c_\ell}J_d, \ell < k$, which inductively have already been estimated, the base case $\ell = 0$ being the standard estimates for Jacobi vector fields as in \cite[Ch.~6.5]{Jos17}. One can then proceed to estimate solutions of this inhomogeneous ODE e.g. by Duhamel's principle.
\end{remark}

\begin{remark}
Let us briefly compare this regularity estimate and proof sketch with another possible approach. First, one switches to harmonic coordinates and uses the assumed injectivity radius and curvature bounds \eqref{e:prelim:exponential:regularity_scale} to obtain $C^{k + 1,\alpha}$ control on the metric in these coordinates for any $0 < \alpha < 1$, following \cite{And90}. Then, one converts this into $C^{k - 1,\alpha}$ control on the metric in exponential coordinates by \cite[Th.~2.1]{DeTKaz81}. The loss of $(1 - \alpha)$ derivatives compared to Lemma \ref{l:exponential_regularity} is immaterial in our context, where we are only ever working on the regularity scale in a smooth manifold.
\end{remark}

Our primary use of the above will be to view the metric $g$ as a form of twisted cone.  Namely, in the context where $\inj(p)>1$ as above we can use exponential coordinates to write the metric $g$ on $B_1(0)\subseteq C(S^{n-1})=\dR^n$ as
\begin{align}\label{e:prelim:exponential:metric_cone}
	g = dr^2 + r^2 g_r\, ,
\end{align}
where $g_r$ is a smooth family of metrics on $S^{n-1}$.  The estimates above can then be understood as estimates on this family $g_r$ :

\begin{corollary}[Cone Regularity of Exponential Map]\label{c:cone_exponential}
		Let $(B_1(p),g)$ be a metric ball which satisfies the regularity scale estimates \eqref{e:prelim:exponential:regularity_scale} with $0 < \eta < 1$.  Let us use exponential coordinates to write $g = dr^2 + r^2 g_r$ as in \eqref{e:prelim:exponential:metric_cone} , where $g_r$ is a family of metrics on $S^{n-1}$.  Then we have the estimates
\begin{align}
	\sum_{\ell =0}^k||\nabla_{g_{S^3}}^{\ell}(g_r-g_{S^3})\,||_{L^\infty(S^3,\, g_{S^3})}&\leq C(n,k)\,\eta\,r^2\, ,\notag\\
	\sum_{\ell =0}^{k-1}||\nabla_{g_{S^3}}^{\ell}g'_r\,||_{L^\infty(S^3,\, g_{S^3})}&\leq C(n,k)\,\eta\,r\, ,\notag\\
	\sum_{\ell =0}^{k-2}||\nabla_{g_{S^3}}^{\ell}g''_r\,||_{L^\infty(S^3,\, g_{S^3})}&\leq C(n,k)\,\eta\,\, .
\end{align}
\end{corollary}

\vspace{.5cm}

%%%%%%%%%%%%%%%%%%%%%%%%%%%%%%%%%%%%%%%%%%%%%%%%%%%%%%%%%%%
\section{Step 1:  The Gluing Block}

In this Section we build the gluing block of Step 1 for our construction.
Our $4$-manifold of interest for this gluing block is the generating line bundle $E\to S^2$, which we can topologically also view as $\dC^2$ blown up at the origin.
We will build a metric of positive Ricci curvature on $E\times S^2$ which has the property that at infinity it looks roughly like $\dR^4\times S^2$.  
More precisely, it will be isometric to the warped product $C(S^3_{1-\epsilon})\times_{\delta r^\alpha} S^2$ near infinity.
The precise Lemma is the following:\\

\begin{lemma}[Inductive Step 1]\label{l:step1:inductive_step1:2}
	For every $0<\epsilon<\frac{1}{10}$, $0<\alpha<\alpha(\epsilon)$ and $0<\delta<\delta(\epsilon,\alpha)$ there exists a smooth Riemannian manifold $\cB(\epsilon,\alpha,\delta)$ with   $\Ric_{\cB}>0$ and such that
\begin{enumerate}
	\item $\cB$ is diffeomorphic to $E\times S^2$, where $E$ is the total space of the generating line bundle $E\to S^2$. 
    Further, $(\cB,g)$ has a warped product structure $g:= g_E+f^2 g_{S^2}$ for some smooth $f:E\to \dR^+$.
	\item $\exists$ $U=U_E\times S^2\subseteq \cB$ such that $\cB\setminus U$ is isometric to a neighborhood of $\infty$ in $C(S^3_{1-\epsilon})\times_{\delta r^\alpha} S^2$ with the metric $dr^2+(1-\epsilon)^2 r^2 g_{S^3}+\delta^2 r^{2\alpha} g_{S^2}$ .
\end{enumerate} 
Further, there exists $\phi:(E,g_E)\to C(S^3_{1-\epsilon})$ such that
\begin{enumerate}
	\item $\phi$ is a diffeomorphism away from the cone point $0\in C(S^3_{1-\epsilon})$, with $\phi^{-1}(0)\cong S^2$ an isometric sphere, 
	\item $|D \phi|\leq C$ is uniformly bounded with $\phi$ an isometry away from $U_E$ .
\end{enumerate}
\end{lemma}
\begin{remark}
	As usual our use of the notation $C(S^3_{1-\epsilon})\times_{\delta r^\alpha} S^2$ means we are looking at the warped product metric on $C(S^3)\times S^2$ given by $g:= dr^2+(1-\epsilon)^2 r^2 g_{S^3}+\delta^2 r^{2\alpha} g_{S^2}$ .
\end{remark}

The proof of the above Lemma is broken down over the remainder of this Section.  
In Step 1.1 of Section \ref{ss:step1.1} we begin by writing down a metric on $(E,g_{E,1})$ with nonnegative Ricci curvature, and which looks like a cone at infinity.  
This cone however may not be close to $\dR^4$ at this stage.

In Step 1.2 of Section \ref{ss:step1.2} we will write down a metric on $E\times S^2$ of the form $g_2=g_{E,2}+f_2^2 g_{S^2}$.  The base metric $g_{E,2} := g_{E,1}$ will simply be the metric from the first step, however we will now equip the metric with a warping $S^2$ factor $f_2(r):= \delta_2 (1+r^2)^{\alpha_2/2}$.  The polynomial growth of the warping factor will add extra curvature which will be useful in flatting out the cone in the third step.

In Step 1.3 of Section \ref{ss:step1.3} we will use the extra curvature provided by the warping factor to slowly increase the cone angle of $g_E$ until it is close to Euclidean.  
In Step 1.4 we will fix the warping factor so that our space becomes isometric to the warped cone $C(S^3_{1-\epsilon})\times_{\delta r^\alpha} S^2$ near infinity.  At several steps we will only build geometries which are globally $C^1$, but we will end the construction of Lemma \ref{l:step1:inductive_step1:2} by applying the $C^1$ smoothing Lemma \ref{l:C1_warped_gluing} in order to fix this issue.

\vspace{.3cm}

\subsection{Step 1.1: Bubble Metric with Positive Ricci}\label{ss:step1.1}

Consider the generating line bundle $E\to S^2$ and let us observe that the unit sphere bundle is diffeomorphic to $S^3$.  In particular, if we remove the zero section then $E\setminus S^2$ is diffeomorphic to $\dR^+\times S^3$.  We will begin by writing our metric in this degenerate coordinate system.  To do so let us choose the canonical left invariant vector fields $X,Y,Z$ on $S^3$, so that they satisfy the commutator relations $[X,Y]=2Z$, $[Y,Z] = 2X$, and $[Z,X] = 2Y$ .  Let $dX,dY,dZ$ denote the dual frames.  We first consider a metric on $\dR^+\times S^3$ of the form

\begin{equation}\label{e:step1.1:metric_ansatz}
	g_{E,1} := dr^2 + A(r)^2 dX^2 + B(r)^2\big(dY^2+dZ^2\big)\, .
\end{equation} 

In order for this to define a smooth metric on $E$ it is required that $A(0)=0$ with $A^{(even)}(0)=0$ and $B(0)>0$ with $B^{(odd)}=0$ .  Our construction for this step is the following:

\begin{lemma}\label{l:step1.1:main}
	Let $g_{E,1}$ be as in \eqref{e:step1.1:metric_ansatz}, and let $0<m<\frac{1}{100}$ with $r_1:= 2$.  Then there exists $A(r),B(r)$ such that
\begin{enumerate}
	% \item $A(r)$, $B(r)$ are smooth away from $r=r_1$ and globally $C^1$ ,
	\item $A(r)=B(r)$ with $A'(r)=B'(r)=m$ for $r\geq r_1$ ,
	\item $g_{E,1}$ defines a $C^1$ metric on $E$, smooth away from $r=r_1$, with $\Ric\geq 0$ on the smooth part, and $\Ric\geq \frac{k^2}{2}>0$ on $U_{r_1}:= \{r< r_1\}$ for some constant $k=k(m)$.
\end{enumerate} 
\end{lemma} 
% \begin{remark}
% 	We will apply the $C^1$ smoothing Lemma \ref{} in order to eventually smooth out near $r=r_1$.
% \end{remark}
\begin{remark}
	Note that for $r\geq r_1$ we have that $g_{E,1} = dr^2 + A(r)^2 g_{S_3}$ is exactly the cone metric on $C(S^3_m)$ .
\end{remark}

\begin{proof}[Proof of Lemma \ref{l:step1.1:main}]
Let $k>0$ be the smallest number satisfying the relation
\begin{equation}
    m=\cos(kr_1)\,.
\end{equation}
Clearly we then have $\frac{\pi}{3}\leq kr_1<\frac{\pi}{2}$.  Let us define $A(r)$ by the formula
\begin{align}
    A(r)=\begin{cases}
    	\frac{1}{k}\sin(kr)&\text{ if } r\leq r_1\,;\\
    	\frac{\sqrt{1-m^2}}{k}+m(r-r_1)&\text{ if } r\geq r_1\, .
    \end{cases}
\end{align}
To define $B(r)$ let $b=b(m)< \frac{\sqrt{1-m^2}}{k}$ be taken so that we can find a smooth function with
\begin{align}
	B(r) = \begin{cases}
 b&\text{ if } r\leq r_1/2\, ,\\
 0\leq B''\leq \frac{4m}{r_1}&\text{ if }r\in [r_1/2,r_1]\,,\\
\frac{\sqrt{1-m^2}}{k}+m(r-r_1)&\text{ if }r\geq r_1\,.
 \end{cases}
\end{align}

Note that in this case we necessarily have $0\leq B'\leq m\leq A'$ and $B\geq A$.  Together with $kr_1<\frac{\pi}{2}$ and $m<\frac{1}{100}$ we can estimate
\begin{align}
    b=B(\frac{r_1}{2})\geq B(r_1)-m(r_1-\frac{r_1}{2})\geq\frac{1}{k}(\sqrt{1-m^2}-\frac{m\pi}{4})>\frac{1}{2k}\, .
\end{align}
Observe from the behavior of $A$ and $B$ as $r\to 0^+$ that $g_{E,1}$ indeed defines a smooth metric on $E$ near the zero section (see e.g. \cite{Per97}).\\

Let us now estimate the Ricci curvatures.
For $r<\frac{r_1}{2}$, we have (note that $B'=0=B''$ here) 
\begin{align}
    \Ric(\partial_r,\partial_r) &= -\frac{A''}{A} - 2\frac{B''}{B}=k^2\,,\notag\\
    \Ric\left(\frac{X}{|X|}, \frac{X}{|X|}\right) &= -\frac{A''}{A} - 2\frac{A'B'}{AB} + 2\frac{A^2}{B^4}\geq k^2\,,\notag\\
    \Ric\left(\frac{Y}{|Y|}, \frac{Y}{|Y|}\right) = \Ric\left(\frac{Z}{|Z|}, \frac{Z}{|Z|}\right) &= -\frac{B''}{B} - \frac{A'B'}{AB} - \left(\frac{B'}{B}\right)^2 +\frac{2}{B^2}+\frac{2B^2 - 2A^2}{B^4}\geq \frac{2}{B(r_1)^2}=\frac{2k^2}{1-m^2}\,,
\end{align}
which is clearly appropriately positive.  For $\frac{r_1}{2}\leq r<r_1$ we can estimate
\begin{align}
    \Ric(\partial_r,\partial_r) &= -\frac{A''}{A} - 2\frac{B''}{B}
    \geq k^2-2\frac{\frac{4m}{r_1}}{\frac{1}{2k}}
    \geq k^2(1-\frac{48m}{\pi})\,,\notag\\
    \Ric\left(\frac{X}{|X|}, \frac{X}{|X|}\right) &= -\frac{A''}{A} - 2\frac{A'B'}{AB} + 2\frac{A^2}{B^4}\geq k^2-2k\cot(k\frac{r_1}{2})\frac{m}{\frac{1}{2k}}+\frac{2k^2\sin^2(k\frac{r_1}{2})}{(1-m^2)^2}\notag\\
    &\geq k^2(1-4\sqrt{3}m+\frac{1}{2(1-m^2)^2})\,,\notag\\
    \Ric\left(\frac{Y}{|Y|}, \frac{Y}{|Y|}\right) = \Ric\left(\frac{Z}{|Z|}, \frac{Z}{|Z|}\right) &= -\frac{B''}{B} - \frac{A'B'}{AB} - \left(\frac{B'}{B}\right)^2 +\frac{2}{B^2}+\frac{2B^2 - 2A^2}{B^4}\notag\\
    &\geq -\frac{\frac{4m}{r_1}}{\frac{1}{2k}}-k\cot(k\frac{r_1}{2})\frac{m}{\frac{1}{2k}}-(\frac{m}{\frac{1}{2k}})^2+\frac{2k^2}{1-m^2}\notag\\
    &\geq k^2(\frac{2}{1-m^2}-\frac{24m}{\pi}-2\sqrt{3}m-4m^2)\,.
\end{align}
Again observe that for $m<\frac{1}{100}$ we have the appropriate positivity.  Finally for $r>r_1$, we have (note that $A = B$ is affine here)
\begin{align}
    \Ric(\partial_r,\partial_r) &=0\,,\notag\\
    \Ric\left(\frac{X}{|X|}, \frac{X}{|X|}\right) &= 2\frac{1-(A')^2}{A^2} =2\frac{1-m^2}{A^2}\,.
\end{align}
Since $m<\frac{1}{100}$ we have $\Ric\geq 0$ for $r> r_1$.
This completes the construction.
\end{proof}

\vspace{.3cm}

\subsection{Step 1.2: Bubble Metric with $S^2$ Warping Factor}\label{ss:step1.2}

Recall we ended the last step by constructing a metric $g_{E,1}$ on $E\to S^2$ which has nonnegative Ricci curvature and is a cone $C(S^3_m)$ outside a compact subset.  The sphere $S^3_m$ in this cone is however potentially quite small, and we will want to take the radius of this sphere closer to $1$ in order to geometrically flatten out the space.\\

In this next step of the construction, we want to add to $g_{E,1}$ a warped $S^2$ factor.  This factor will add additional curvature to the radial directions which will be used in subsequent sections to flatten out our cone structure.  In this Step we will look for a metric of the form
\begin{align}\label{e:step1.2:metric_ansatz}
	g_{2} &:= g_{E,2}+f_2(r)^2 g_{S^2}\notag\\
	&= g_{E,1}+f_2(r)^2 g_{S^2}\, .
\end{align}

In particular, we will not change the metric on our base $E$ in this Step.  The warping factor $f_2(r)$ will be given explicitly by
\begin{align}\label{e:step1.2:metric_ansatz2}
	f_2(r)=\delta_2(1+r^2)^\frac{\alpha_2}{2}\, .
\end{align}

Our main purpose in this Step is to see that $g_2$ always has positive Ricci curvature:

\begin{lemma}\label{l:step1.2:main}
    Let $g_2$ be as in \eqref{e:step1.2:metric_ansatz} and \eqref{e:step1.2:metric_ansatz2}.
    Then for any $0<\alpha_2\leq \alpha_2(m)\leq \frac{1}{2}$ and $0<\delta_2<1$,
\begin{equation}\label{metric_bubble_with_S2_part}
    g_2=g_{E,2}+f_2^2g_{S^2}=dr^2+A(r)^2dX^2+B(r)^2(dY^2+dZ^2)+f_2(r)^2g_{S^2}
\end{equation} 
defines a $C^1$ metric on $E\times S^2$, smooth away from $r=r_1$ with $\Ric>0$ on the smooth part.
\end{lemma}

\begin{proof}
Note first that as $f_2^{(\text{odd})}(0)=0$, the metric is smooth near $r=0$.  Calculations together with the results from the previous subsection give that for $r <r_1$:
\begin{align}
    \frac{f'_2}{f_2}&=\frac{\alpha_2 r}{1+r^2}\leq \frac{\alpha_2}{2}\,,\notag\\
    \frac{f''_2}{f_2}&=\frac{\alpha_2}{1+r^2}(1+\frac{(\alpha_2-2)r^2}{1+r^2})\leq  \alpha_2\,,\notag\\
    \frac{A'f'_2}{Af_2}&=\frac{k\cos(kr)\alpha_2 r}{\sin(kr)(1+r^2)}\leq \frac{\alpha_2 \cos(kr)}{1+r^2}\frac{\pi}{2}\leq \frac{\pi}{2}\alpha_2\,,\notag\\
    \frac{B'f'_2}{Bf_2}&\leq \frac{m}{B(\frac{r_1}{2})}\frac{\alpha_2}{2}\leq mk\alpha_2\, .
\end{align}
We now estimate the Ricci curvature of the whole space for $r<r_1$ (with Lemma \ref{l:step1.1:main}):
\begin{align}
\Ric(\partial_r,\partial_r) &= -\frac{A''}{A} - 2\frac{B''}{B} -2\frac{f''_2}{f_2}\geq \frac{k^2}{2}-2\alpha_2\,,\notag\\
\Ric\left(\frac{X}{|X|}, \frac{X}{|X|}\right) &= -\frac{A''}{A} - 2\frac{A'B'}{AB} + 2\frac{A^2}{B^4} -2\frac{A'f'_2}{Af_2} \geq \frac{k^2}{2}-\alpha_2 \pi\,,\notag\\
\Ric\left(\frac{Y}{|Y|}, \frac{Y}{|Y|}\right) = \Ric\left(\frac{Z}{|Z|}, \frac{Z}{|Z|}\right) &= -\frac{B''}{B} - \frac{A'B'}{AB} - \left(\frac{B'}{B}\right)^2 + 2\frac{2B^2 - A^2}{B^4}-2\frac{B'f'_2}{Bf_2}\geq \frac{k^2}{2}-2mk\alpha_2\,,\notag\\
f^{-2}\Ric(\partial_\alpha, \partial_\alpha)&=\frac{1}{f^2_2}-\frac{f''_2}{f_2}-\frac{(f')^2}{f^2}-\frac{A'f'_2}{Af_2}-2\frac{B'f'_2}{Bf_2}\notag\\
&\geq \frac{1}{\delta_2^2(1+r_1^2)^{\alpha_2}}-\alpha_2-\frac{\alpha_2^2}{4}-\frac{\pi}{2}\alpha_2-2mk\alpha_2\,.
\end{align}

Observe for any $0<\alpha_2\leq \alpha_2(m)$ and any $0<\delta_2<1$, we have $\Ric>0$. 
For $r>r_1=2$ we use different estimates (recall that $A=B$ and they are affine here, and we impose $\alpha_2\leq\frac{1}{2}$):
\begin{align}
    \frac{f'_2}{f_2}&=\frac{\alpha_2 r}{1+r^2}\,,\notag\\
    \frac{f''_2}{f_2}&=\frac{\alpha_2}{1+r^2}(1+\frac{(\alpha_2-2)r^2}{1+r^2})
    \leq \frac{\alpha_2}{1+r^2}(1+\frac{(\alpha_2-2)r_1^2}{1+r_1^2})
    \leq \frac{-\alpha_2}{5(1+r^2)}\,,\notag\\
    \frac{A'f'_2}{Af_2}&=\frac{m\alpha_2 r}{A(1+r^2)}\,.
\end{align}
Thus we can estimate the Ricci curvatures for $r>r_1$:
\begin{align}
\Ric(\partial_r,\partial_r) &=-3\frac{A''}{A}-2\frac{f''_2}{f_2}>0\,,\notag\\
\Ric\left(\frac{X}{|X|}, \frac{X}{|X|}\right) &= -\frac{A''}{A}+ 2\frac{1-(A')^2}{A^2} -2\frac{A'f'_2}{Af_2}=2\frac{1-m^2}{A^2}-2\frac{m\alpha_2 r}{A(1+r^2)}\geq \frac{2}{Ar}(kr_1\sqrt{1-m^2}-m\alpha_2)\,,\notag\\
f^{-2}\Ric(\partial_\alpha, \partial_\alpha)&=\frac{1}{f^2_2}-\frac{f''_2}{f_2}-\frac{(f_2')^2}{f_2^2}-3\frac{A'f'_2}{Af_2}\geq\frac{1}{\delta_2^2(1+r^2)^{\alpha_2}}+\frac{\alpha_2}{5(1+r^2)}-(\frac{\alpha_2 r}{1+r^2})^2-3\frac{m\alpha_2 r}{A(1+r^2)}\notag\\
&\geq \frac{1}{1+r^2}\left(\frac{1}{\delta_2^2}+\frac{\alpha_2}{5}-\alpha_2^2-3\alpha_2\right)\,.
\end{align}
Thus we again have for any $0<\alpha_2\leq \alpha_2(m)$ and $0<\delta_2<1$ that $\Ric>0$. 
\end{proof}

\vspace{.3cm}

\subsection{Step 1.3 Flattening the Cone}\label{ss:step1.3}

Our goal in this Step of the construction is to look for a metric on $E\times S^2$ of the form
\begin{align}\label{e:step1.3:metric_ansatz}
	g_3:= g_{E,3}+ f_3(r)^2 g_{S^2}\, ,
\end{align}
where the warping factor $f_3(r):= f_2(r) = \delta_2(1+r^2)^{\alpha_2/2}$ will remain unchanged, up to further restrictions on the parameters $\delta_2$ and $\alpha_2$.  The base metric $g_{E,3}$ should look like a flat cone $C(S^3_{1-\epsilon})$ outside some large radius, and will more generally satisfy 
\begin{align}\label{e:step1.3:metric_ansatz2}
	g_{E,3} = \begin{cases}
 	g_{E,2}&\text{ if } r\leq r_1\,,\\
 	dr^2+h_3(r)^2 g_{S^3}&\text{ if } r\geq r_1\,,
 \end{cases}
\end{align}
where the warping function $h_3$ will be smooth on $[r_1,\infty)$ and satisfy 
\begin{align}\label{e:step1.3:metric_ansatz3}
	\begin{cases}
h_3(r_1)=A(r_1)\,,\quad h'_3(r_1)=m\,,\\
0\leq r h_3''\leq 10/\ln r_3 &\text{ if } r\in [r_1,r_3]\,,\\
h'_3(r)=1-\epsilon &\text{ if }r\geq r_3\,.
\end{cases}
\end{align}
The following lemma will tell us that for $r_3$ sufficiently large the resulting metric will have positive Ricci curvature:

\begin{lemma}\label{l:step1.3:main}
	Let $g_3$ satisfy \eqref{e:step1.3:metric_ansatz}, \eqref{e:step1.3:metric_ansatz2}, \eqref{e:step1.3:metric_ansatz3} with $\alpha_2<\alpha_2(\epsilon,m)\leq \frac{1}{2}$ , $r_3\geq r_3(m,\alpha_2,\epsilon)$ and $\delta_2<\delta_2(m)$.  Then $g_3$ is smooth away from $r=r_1$, globally $C^1$, and satisfies $\Ric> 0$ on the smooth region.
\end{lemma}
\begin{proof}
	We will focus our computations in the range $r\in [r_1,r_3]$, as in the range $r\in [r_3,\infty)$ the metric $g_{E,3}$ is again conic and the estimate will be similar as in the previous subsection.  Let us begin by computing the Ricci curvature of the ansatz $g_3= dr^2+h_3(r)^2 g_{S^3}+f_3(r)^2 g_{S^2}$ as:
\begin{align}\label{e:step1.3:Ricci}
	\Ric_{rr} &= -3\frac{h_3''}{h_3}-2\frac{f_3''}{f_3}\, ,\notag\\
	\Ric_{ii} & = 2\frac{1-(h'_3)^2}{h_3^2}-\frac{h_3''}{h_3}-2\frac{h_3'}{h_3}\frac{f_3'}{f_3} \, ,\notag\\
	\Ric_{\alpha\alpha} & = \frac{1-(f'_3)^2}{f_3^2}-\frac{f_3''}{f_3}-3\frac{f_3'}{f_3}\frac{h_3'}{h_3} \, .
\end{align}

Let us impose the restriction $\alpha_2\leq \frac{1}{2}$ and calculate $mr_1\leq A(r_1)\leq \frac{9}{10}r_1\leq (1-\epsilon)r_1$.  Then in the range $r\in[r_1,r_3]$, let us observe the estimates:
\begin{align}
	\frac{f'_3}{f_3}\frac{h'_3}{h_3}&\leq \frac{\alpha_2r}{1+r^2}\frac{1-\epsilon}{m(r-r_1)+A(r_1)}\leq \frac{\alpha_2}{m}\frac{1}{r^2}\,,\notag\\
	\frac{1-(h_3')^2}{(h_3)^2}&\geq \frac{1-(1-\epsilon)^2}{((1-\epsilon)(r-r_1)+A(r_1))^2}\geq \frac{\epsilon}{r^2}\,,\notag\\
	\frac{1-(f_3')^2}{(f_3)^2}&=\frac{1}{\delta_2^2(1+r^2)^{\alpha_2}}-\frac{\alpha_2^2r^2}{(1+r^2)^2} \geq \frac{1}{\delta_2^2r^{2\alpha_2}}\left(\frac{r^2}{1+r^2}\right)^{\alpha_2}-\frac{1}{4r^2}\geq\frac{1}{2\delta_2^2r^{2\alpha_2}}\geq \frac{1}{2\delta_2^2r}\,,\notag\\
	\Big|\frac{h''_3}{h_3}\Big|&\leq \frac{10}{m \ln r_3}\frac{1}{r^2}\,,\notag\\
	\frac{f''_3}{f_3} &=\frac{\alpha_2}{1+r^2}(1+\frac{(\alpha_2-2)r^2}{1+r^2})
    \leq \frac{-\alpha_2}{5(1+r^2)}\leq - \frac{4\alpha_2}{25}\frac{1}{r^2}\, . 
\end{align}
If we plug these estimates into \eqref{e:step1.3:Ricci} then we arrive at
\begin{align}
	\Ric_{rr}&\geq \Big(\frac{8\alpha_2}{25} - \frac{30}{m\ln r_3}\Big)\frac{1}{r^2}\, ,\notag\\
	\Ric_{ii}&\geq  \Big(2\epsilon-\frac{10}{m\ln r_3}-\frac{2\alpha_2}{m}\Big)\frac{1}{r^2}\, ,\notag\\
	\Ric_{\alpha\alpha} &\geq \frac{1}{2\delta_2^2 r}+\Big(\frac{4}{25} -\frac{3}{m}\Big)\frac{\alpha_2}{r^2}
    \geq \frac{1}{2\delta_2^2 r}+\frac{1}{2}\Big(\frac{4}{25} -\frac{3}{m}\Big)\frac{1}{r}\, .
\end{align}
It follows from the first inequality that if $r_3\geq r_3(m,\alpha_2)$ then $\Ric_{rr}>0$.  It follows from the second inequality that if $\alpha_2\leq \alpha_2(m,\epsilon)$ and $r_3\geq r_3(m,\epsilon)$ then $\Ric_{ii}>0$.  Finally we see from the last equation that if $\delta_2\leq \delta_2(m)$ that $\Ric_{\alpha\alpha}>0$.

\end{proof}

\vspace{.3cm}

\subsection{Step 1.4:  The Warped Cone Metric}\label{ss:step1.4}

In the last step of the construction we have built a global metric $g_3$ on $E\times S^2$ such that outside the compact set $U_3:= \{r\leq r_3\}$, the metric $g_3$ can be written as
\begin{align}
	g_3 = dr^2+(1-\epsilon)^2 (r-R_3)^2 g_{S^3} +\delta_2^2 (1+r^2)^{\alpha_2} g_{S^2}\, ,
\end{align}
where $R_3$ solves $h(r_3) = (1-\epsilon)(r_3-R_3)$.  Observe that $R_3>0$ under our assumptions of the parameters.  Our goal in this Step of the construction is to build a metric $g_4$ which agrees with $g_3$ for $r\leq r_3$, but for $r\geq r_3$ should take the form
\begin{align}\label{e:step1.4:metric_ansatz}
	g_4 = dr^2+(1-\epsilon)^2 (r-R_3)^2 g_{S^3} +f_4(r)^2 g_{S^2}\, ,
\end{align}
where our warping factor satisfies
\begin{align}\label{e:step1.4:metric_ansatz2}
	f_4(r) = \begin{cases}
 	f_3(r)&\text{ if } r\leq r_3\, ,\\
 	\delta (r-R_3)^{\alpha}&\text{ if } r\geq r_3\, .
 \end{cases}
\end{align}
We want our warping function to be globally $C^1$, and thus we will choose the constants
\begin{align}
	\alpha:=\alpha_2\frac{r_3\frac{h(r_3)}{1-\epsilon}}{1+r_3^2}<\alpha_2\,,\quad \delta:= \delta_2\frac{ (1+r_3^2)^\frac{\alpha_2}{2}}{(\frac{h(r_3)}{1-\epsilon})^\alpha}\,.
\end{align}

Our final Lemma is that for $\alpha_2$ and $\delta_2$ sufficiently small our new metric $g_4$ has positive Ricci curvature:

\begin{lemma}
	Let $g_4$ satisfy \eqref{e:step1.4:metric_ansatz} and \eqref{e:step1.4:metric_ansatz2}. If we further choose $\alpha_2\leq \alpha_2(\epsilon)$ and $\delta_2\leq \delta_2(\alpha_2)$ then for $r>r_3$ we have that $\Ric>0$.
\end{lemma}
\begin{remark}
    Notice that after the change of coordinates $t:= r-R_3$, the metric $g_4$ above becomes the desired format
\begin{equation}
    dt^2+((1-\epsilon)t)^2g_{S^3}+(\delta t^\alpha)^2g_{S^2}
\end{equation}
as in Lemma \ref{l:step1:inductive_step1:2}.
\end{remark}
\begin{proof}
The range $r>r_3$ corresponds exactly to $t>\frac{h(r_3)}{1-\epsilon}$, and in these new coordinates we can compute the Ricci curvature of $g_4$ as
\begin{align}
\Ric_{tt} &= -2\frac{\alpha(\alpha-1)}{t^2}>0\,, \notag\\
h^{-2}\Ric_{ii} &=\Big(\frac{1}{(1-\epsilon)^2}-1-\alpha\Big)\frac{2}{t^2}\,,  \notag\\
f^{-2}\Ric_{\alpha\alpha} &= \Big(\alpha(1-\alpha)+\frac{t^{2-2\alpha}}{\delta^2}-\alpha^2-3\alpha\Big)\frac{1}{t^2}\geq \Big(\alpha(1-\alpha)+\frac{(\frac{h(r_3)}{1-\epsilon})^{2-2\alpha}}{\delta^2}-\alpha^2-3\alpha\Big)\frac{1}{t^2}\, .
\end{align}
Notice that as $\alpha_2\to 0$ we have that $\alpha\to 0$, and similarly (after fixing $\alpha_2, \alpha, r_3$) we have that $\delta_2\to 0$ as $\delta\to 0$.   Thus for $\alpha_2\leq \alpha_2(\epsilon)$ the second term is uniformly positive.  Finally for $\delta_2\leq \delta_2(\alpha_2)$ we have that the third term is uniformly positive.
\end{proof}

\vspace{.3cm}

\subsection{Finishing the Proof of Lemma \ref{l:step1:inductive_step1:2}}\label{ss:step1.5}

For $\epsilon>0$ and $m=\frac{1}{10^3}$ fixed we can now choose $\alpha<\alpha(\epsilon)$ and $\delta<\delta(\epsilon, \alpha)$.
Let us now equip $E\times S^2$ with the metric $g_4$.  Recall that this metric is smooth away from $r\in\{r_1,r_3\}$, globally $C^1$ and satisfies $\Ric>0$ on the smooth part.  We can now apply the $C^1$ smoothing Lemma \ref{l:C1_warped_gluing} in order to build a smooth metric $g=g_E+f^2 g_{S^2}$ on $E\times S^2$ with $\Ric>0$ such that $g=g_4$ for $r\geq 2r_3$.  This completes the construction of $\cB=\cB(\epsilon,\alpha,\delta)$.  What remains is to define and study the projection map $\phi:(E,g_E)\to C(S^3_{1-\epsilon})$.\\

Recall that we have coordinates $(r,\omega)$ on $\dR^+\times S^3$, and we have identified $E\setminus S^2$ and $C(S^3)\setminus\{0\}$ with $\dR^+\times S^3$.  Our mapping $\phi:E\to C(S^3)$ will then take the form
\begin{align}
	\phi(r,\omega)=(\lambda(r), \omega)\, ,
\end{align}
where $\lambda:[0,\infty)\to [0,\infty)$ is a smooth function with 
\begin{equation}
    \begin{cases}
        \lambda(0)=0\,,\\
    \lambda', \lambda''    
    >0\,,\\
    \lambda(r_3)=r_3-R_3< r_3\,,\\
    \lambda'(r)=1\quad\text{for $r\geq r_3$}\, .
    \end{cases}
\end{equation}

Note that since $\lambda'\leq 1$ and $\lambda(0)=0$, we have $\lambda(r)\leq r$ for all $r>0$.  Also observe that $\phi$ sends the zero section $S^2$ of $E$ to the cone point $0$ of $C(S^3_{1-\epsilon})$.  It then suffices to estimate $|D\phi|$ in terms of the metrics on $E$ and $C(S^3_{1-\epsilon})$.  Note that as the $C^1$ gluing lemma produces a smooth metric on $E$ which is an arbitrarily small $C^1$ perturbation, it is enough to estimate in the metric $g_4$ on $E$.\\

For $r< r_1$, we have (using the notation and results in Lemma \ref{l:step1.1:main})
\begin{align}
    |D\phi(\partial_r)|&=\lambda'(r)\leq 1\,;\notag\\
    |D\phi(\frac{X}{|X|})|&=\frac{(1-\epsilon)\lambda(r)}{A(r)}\leq \frac{(1-\epsilon)r}{\frac{1}{k}\sin(kr)}\leq \frac{(1-\epsilon)kr_1}{\sin(kr_1)}\leq \frac{\pi}{2}(1-\epsilon)\,;\notag\\
    |D\phi(\frac{Y}{|Y|})|&=|D\phi(\frac{Z}{|Z|})|=\frac{(1-\epsilon)\lambda(r)}{B(r)}\leq \frac{(1-\epsilon)r_1}{B(\frac{r_1}{2})}\leq \pi (1-\epsilon)\,.
\end{align}
For $r_1<r<r_3$, we have (using the notation and results in Lemma \ref{l:step1.3:main})
\begin{align}
    |D\phi(\partial_r)|&=\lambda'(r)\leq 1\,;\notag\\
    |D\phi(\frac{\partial \omega}{|\partial \omega|})|&=\frac{(1-\epsilon)\lambda(r)}{h_3(r)}\leq \frac{(1-\epsilon)r }{A(r_1)+(r-r_1)m}\leq \frac{1-\epsilon}{m}
\end{align}
Since $m=10^{-3}$ is chosen universally, the bounds above do not depend on any other parameters.\\

For $r>r_3$, $D\phi$ is an isometry.  This concludes the proof of Lemma \ref{l:step1:inductive_step1:2}. $\qed$
\vspace{.5cm}

%%%%%%%%%%%%%%%%%%%%%%%%%%%%%%%%%%%%%%%%%%%%%%%%%%%%%%%%%%%%%%%%%%%%%%%%%%%
\vspace{.5cm}

\section{Step 2:  Adding Conical Singularities}\label{s:step2}

We complete Step 2 of the construction in this Section.  Namely, we want to see how to take a manifold $M=X^4\times_{f} S^2$ and add a cone point in any arbitrarily small neighborhood of $X^4$ while (almost) preserving a Ricci curvature lower bound.  Our primary setup, essentially after rescaling on the regularity scale of $M$, is to assume we are faced with a warped product space $(B_2(p)\times S^2,g)$ with metric $g=g_B+f^2 g_{S^2}$, under the assumptions
\begin{align}\label{e:step2:regularity_scale}
\Ric_g >\lambda g\, ,\inj_{g_B}(p) &>2\, , \notag\\
	|\Rm_{g_B}|_{g_B}\, ,\, |\nabla\Rm_{g_B}|_{g_B}\, ,\, |\nabla^2\Rm_{g_B}|_{g_B}&<\eta < 1 \, ,\notag\\
	|\nabla_{g_B}\ln f|_{g_B}\, ,\, |\nabla^2_{g_B}\ln f|_{g_B}\, ,\, |\nabla^3_{g_B} \ln f|_{g_B} &< \eta \, .
\end{align}

Note that there are no assumptions about the sign of $\lambda \in \dR$. Observe that the above hold for any warped product $M^4\times_f S^2$ so long as we work on the regularity scale.  Our main result in this Section is the following:\\

\begin{lemma}[Inductive Step 2]\label{lem:inductive_step2:2}
	Consider a warped product space $(B_2^4(p) \times S^2,g)$ with metric $g = g_B + f^2g_{S^2}$ satisfying \eqref{e:step2:regularity_scale}, and write $r := \dist_{g_B}(\cdot,p)$.
Then for all choices of parameters $0 < \eps < \eps(|\lambda|)$, $0<\alpha < \alpha(\eps)$, $0 < \hat{r} < \hat r(\alpha, \lambda,\eps)$, and $0 < \hat{\delta} < \hat{\delta}(\lambda,\alpha,\epsilon,||f||_{L^\infty},\hat{r})$, there exists $0 < \delta = \delta(\,\hat{\delta}\|f\|_\infty\mid \alpha, \lambda, \eps)$ and a warped product metric $\hat g = \hat{g}_B+\hat f^{\,2} g_{S^2}$ such that:
\begin{enumerate}
	\item The Ricci lower bound $\Ric_{\hat g}>\lambda - C(4)\eps$ holds for $\hat r/2 \leq r \leq 2$\, ,
	\item $\hat g = g_B+\hat{\delta}^2 f^2 g_{S^2}$ is unchanged up to scaling the warping factor $f$ by $\hat \delta$ for $1 \leq r \leq 2$ ,
	\item $\hat g = dr^2+(1-\eps)^2 r^2 g_{S^3}+\delta^2 r^{2\alpha} g_{S^2}$ has the cone warping structure $C(S^3_{1 - \eps}) \times_{\delta r^\alpha}S^2$ for $r\leq \hat r$ ,
	\item The identity map $\Id:(B_2(p),\hat{g}_B)\to (B_2(p),g_B)$ is $(1+2\epsilon)$-bi-Lipschitz.
\end{enumerate}
\end{lemma}

\begin{remark}
  Most of the results of this section hold in general dimensions, where the constants should then include a dimensional dependence.  Our notation $C(4)$ denotes that the constant only depends on the dimension $n=4$.
  
\end{remark}
\begin{remark}
	Recall that our notation of the constant dependence $\delta = \delta(\,\hat{\delta}\|f\|_\infty\mid \alpha, \lambda, \eps)$ means that $\delta\to 0$ as $\hat{\delta}\|f\|_\infty\to 0$ with the other constants $\alpha,\lambda,\epsilon$ fixed .
\end{remark}

The construction will be broken down into three steps. In Step 2.1 of Section \ref{ss:step2:1} we begin by writing the metric $g_B=dr^2 +r^2 g_r$ in exponential polar coordinates, where $g_r$ is a smooth family of metrics on $S^3$ which naturally converges to the standard metric as $r\to 0$.  Our primary goal in Step 2.1 is to alter the base metric $g_B$ to a metric $g_{B,1}$.  The metric $g_{B,1}$ will agree with $g_B$ for $1\leq r\leq 2$, however it will take the form  $g_{B,1}=dr^2 +(1-\epsilon)^2r^2 g_r$ for $r\leq r_1$.  This will give $g_{B,1}$ a large amount of additional positive Ricci curvature in the nonradial directions, which we will exploit in future steps.  Additionally, we are able to ensure that $\Ric$ of the total space drops by at worst an $\eps$-small amount when passing from $g$ to $g_1$. \\

In Step 2.2 of Section \ref{ss:step2:2} we focus on the $S^2$ warping factor and leave the base $g_{B,2}:= g_{B,1}$ fixed.  Our goal will be to construct a warping factor $f_2$ so that $f_2= \hat{\delta} f_1$ for $r\geq r_1$, while $f_2 = \delta\, r^\alpha$ for $r\leq r_2$.  The effect of this will be to add a large amount of Ricci curvature to the radial direction of $g_{2}:= g_{B,2}+f_2^2 g_{S^2}$.\\

In the final Step 2.3 of Section \ref{ss:step2:3} we will use the additional positive Ricci curvature introduced in the first two steps to once again alter the base metric $g_{B,3}$, while fixing $f_3 := f_2$.  We will preserve $g_{B,3}=g_{B,2}$ for $r\geq r_2$, however for $r\leq r_3$ we will ensure that $g_{B,3} = dr^2 + (1-\epsilon)^2r^2 g_{S^3}$ is the standard cone $C(S^3_{1-\epsilon})$.  This will complete the construction of $\hat{g}$, and in Section \ref{ss:step2:4} we will check the final bi-Lipschitz property of the construction.

\vspace{.3cm}

\subsection{Step 2.1: Decreasing the Cone Angle}\label{ss:step2:1}

Consider the metric $g_B$ on $B_2(p)$, and by using the radial function $r := d(\cdot, p)$ and exponential coordinates let us write $g_B$ as 
\begin{align}
	g_B = dr^2 +r^2 g_{r}\, ,
\end{align}
where $g_r$ is a smooth family of metrics on $S^3$.  It follows from Corollary \ref{c:cone_exponential}, recalling that $\eta<1$, that we have the estimates
\begin{align}\label{e:step2:1:metric_decay_estimates}
	\big(2 - C(4)\eta\, r^2\big)g_r &\leq \Ric_{g_r}\, ,\notag\\
	\|g'_r\|_{L^\infty(S^3, g_r)}&\leq C(4)\eta\, r \, .
\end{align}

Let us remark that the first line of \eqref{e:step2:1:metric_decay_estimates} follows either from the $C^2$ estimate Corollary \ref{c:cone_exponential} or directly from the second line of \eqref{e:step2:1:metric_decay_estimates} by applying the Gauss equations and the identity $\frac{1}{2}(r^2g_r)' = II^{\partial B_r(p)}$.\\

 In this subsection we will look for a metric 
\begin{align}
	g_1 := g_{B,1}+ f_1^2 g_{S^2}\, ,
\end{align}
under the ansatz 
\begin{align}\label{e:step2:1:metric_ansatz}
	g_{B,1} &= dr^2 +h(r)^2 g_r \, ,\notag\\
	f_1 & := f\, .
\end{align}

Observe that $g_r$ is the original family of metrics on $S^3$ and that $f_1$ is a function on $B_2(p)$.  For $r_1 := 1/2 $ , the function $h(r)$ will be chosen as any smooth function with the properties that 
\begin{align}\label{e:step2:1:metric_ansatz:2}
	\begin{cases}
		h(r) = r &\text{ if } r\geq 1\, ,\\
		|h - r| + |h'-1|+|h''|\leq 10\epsilon &\text{ if } r_1\leq r\leq 1\,,\\
		h(r) = (1-\epsilon)r &\text{ if } r\leq r_1\, .
	\end{cases}
\end{align}

In particular, this construction implies that $g_{B,1} = g_B$ for $1 \leq r \leq 2$, while $g_{B,1} = dr^2 + (1 - \eps)^2g_r$ for $r \leq r_1$. Note that on $B_{r_1}(p)$ we have introduced a cone singularity at $p$, and we will see that this introduces a scale invariant (positive) blow up of the Ricci curvature near $p$.  Our main result in this subsection is that for $\epsilon$ sufficiently small, the metric $g_1$ has Ricci curvature that drops by an arbitrarily small amount from that of $g$:\\

\begin{lemma}\label{l:step2:1:main}
Let $g$ satisfy the assumptions of Lemma \ref{lem:inductive_step2:2} with $g_{B,1}$ defined as in \eqref{e:step2:1:metric_ansatz} and \eqref{e:step2:1:metric_ansatz:2}.  Then for any $0 < \epsilon<\epsilon(|\lambda|)$, we have that
\begin{enumerate}
\item $\left.\Ric_{g_1}\right|_{TB_2(p)}>\lambda - C(4)\,\eps$ on $\big(B_2(p)\setminus \{p\}\big)\times S^2$ ,
\item $\left.\Ric_{g_1}\right|_{TS^2} > f^{-2} - C(4)\eta\, r^{-1}$ on $\big(B_2(p)\setminus \{p\}\big)\times S^2$ .
\end{enumerate}
\end{lemma}

\begin{remark}
The reader may wonder at the disagreeable $\left.\Ric_{g_1}\right|_{TS^2}$ estimate. This is simply an artifact of the division of the proof of Lemma \ref{lem:inductive_step2:2} into distinct steps. As soon as the radius $r_2$ is chosen in Step \ref{ss:step2:2}, we will multiply $f$ by the small but positive number $0 < \hat{\delta} = \hat{\delta}(r_2,\ldots)$ so that the first term of the above $\left.\Ric_{g_1}\right|_{TS^2}$ lower  bound dominates the second in the range $r_2 \leq r \leq 2$.
\end{remark}

\begin{proof}

Our standard notation for the rest of this Section will be to use $i,j,\ldots$ to represent coordinate directions on $S^3$ and $\alpha,\beta,\ldots$ to represent coordinate directions on $S^2$.  The index $r$ is reserved for the radial direction, and we will sometimes write $\partial_r h = h'$ to represent radial derivatives.  Let us compute the Ricci curvature of the ansatz \eqref{e:step2:1:metric_ansatz} as follows:

\begin{align}\label{e:step2:1:Ricci:1}
(\Ric_{g_1})_{rr} &= (\Ric_g)_{rr} - 3\frac{h''}{h} +\left(\frac{1}{r} -\frac{h'}{h}\right)\tr_{g_r}(g_r')\, , \notag\\
(\Ric_{g_1})_{ir} &= (\Ric_g)_{ir} + \frac{2}{h}\left(\frac{h'}{h} - \frac{1}{r}\right)\frac{h\partial_i f_1}{f_1}\, , \notag\\
(\Ric_{g_1})_{ij} &= (\Ric_g)_{ij} + \frac{h^2 - r^2}{r^2}(\Ric_g)_{ij} + \left(\frac{1}{h^2} - \frac{1}{r^2}\right)(h^2\Ric_{g_r})_{ij} - \frac{h''}{h} (g_1)_{ij} + 2\left(\frac{1}{r^2} - \frac{(h')^2}{h^2}\right)(g_1)_{ij}  \notag \\
&+ \left(\frac{1}{r} - \frac{h'}{h}\right)\left(\frac{1}{2}\tr_{g_r}(g_r')(g_1)_{ij} + \frac{3}{2}(h^2g_r')_{ij} + 2\frac{f'_1}{f_1}(g_1)_{ij}\right) + 2\frac{h^2 - r^2}{r^2}\left(\frac{(\nabla^2_{g_B} f_1)_{ij}}{f_1} -\frac{1}{2}\frac{f_1'}{f_1}(g_r')_{ij}\right)\, , \notag\\
(\Ric_{g_1})_{\alpha\beta} &= \left[\frac{1}{f^2_1} -\frac{r^2}{h^2}\left(\frac{\Delta_{g_B} f_1}{f_1} + \frac{|\nabla^{g_B} f_1|^2_{g_B}}{f^2_1}\right) + \frac{r^2 - h^2}{h^2}\left( \frac{f''_1}{f_1} + \frac{(f'_1)^2}{f^2_1} +\frac{1}{2}\frac{f'_1}{f_1}\tr_{g_r}(g_r') \right) \right.\notag\\
&\left.+ 3\left(\frac{r^2 - h^2}{rh^2} + \frac{1}{r} - \frac{h'}{h}\right)\frac{f'_1}{f_1} \right](g_1)_{\alpha\beta}\, .
\end{align}

As in \eqref{e:step2:1:metric_decay_estimates} we have that

\begin{equation}
\big(2 - C\eta\, r^2\big)g_r \leq \Ric_{g_r}\, ,\qquad\|g'_r\|_{L^\infty(S^3, g_r)}\leq C\eta\, r \, .
\end{equation}

Additionally, we recall the assumptions

\begin{equation}
|\nabla_{g_B} \ln f|_{g_B} < \eta\, ,\qquad \left|\frac{(\nabla^2_{g_B} f)}{f}\right|_{g_B} \leq |\nabla^2_{g_B} \ln f|_{g_B} + |\nabla_{g_B} \ln f|^2_{g_B} < C\eta\, .
\end{equation}

Let us focus first on the region $r\in [r_1,1]$, which is the worst of the two regions.  If we plug the above estimates and \eqref{e:step2:1:metric_ansatz:2} into \eqref{e:step2:1:Ricci:1}, we arrive at
\begin{align}\label{e:step2:1:Ricci:2}
(\Ric_{g_1})_{rr} &> \lambda   - C(1+ C\eta) \epsilon\, , \notag\\
\big|(\Ric_{g_1} - \Ric_g)_{ir}\big|_{(g_1)_{ij}} &<  C\eta\epsilon \, , \notag\\
(\Ric_{g_1})_{ij} &>  \Big(\lambda  -C(1+C\eta)\epsilon\Big)(g_1)_{ij}\, ,\notag\\
(\Ric_{g_1})_{\alpha\beta} &> \left(\frac{1}{f^2} - C\eta\right)(g_1)_{\alpha\beta}\, . 
\end{align}

We see from the first three inequalities that if $\epsilon<\epsilon(|\lambda|)$, then $\left.\Ric_{g_1}\right|_{TB_2(p)}>\lambda - C(4)\eps$.  If we focus now on $r\leq r_1$, where $h(r) = (1 - \eps)r$, then most of the error terms of \eqref{e:step2:1:Ricci:1} vanish. What remains are the estimates 
\begin{align}\label{e:step2:1:Ricci:3}
(\Ric_{g_1})_{rr} &> \lambda  \, , \notag\\
(\Ric_{g_1})_{ri} &= (\Ric_g)_{ri}\, , \notag\\
(\Ric_{g_1})_{ij} &> \left(\lambda +\frac{(2-C\eta r^2)\epsilon}{r^2} - C\frac{\eta\eps}{r}-C\eta\eps\right)(g_1)_{ij}\, ,\notag\\
(\Ric_{g_1})_{\alpha\beta} &>  \left(\frac{1}{f^2} - C\frac{\eta\eps}{r} - C\eta\right)(g_1)_{\alpha\beta}\, .
\end{align}

We again see that for $\epsilon<\epsilon(|\lambda|)$, the Ricci curvature satisfies $\Ric_{g_1}\geq \lambda - C\eps$ in the $TB_2(p)$ directions. We also see that $\left.\Ric_{g_1}\right|_{TS^2} > f^{-2} - C\eta r^{-1}$ holds for both the $r_1 \leq r  \leq 1$ region and the $r \leq r_1$ region.
\end{proof}

\vspace{.3cm}

\subsection{Step 2.2: The $S^2$ Warping Factor}\label{ss:step2:2}

We now want to alter the metric $g_1$ in the range $r_2\leq r\leq r_1=\frac{1}{2}$.  We are looking for a metric $g_2$ of the form

\begin{align}\label{e:step2.2:metric_ansatz}
	g_2 := g_{B,2}+ f_2(x)^2 g_{S^2} := g_{B,1}+ f_2(x)^2 g_{S^2}\, .
\end{align}
In particular, we will not alter the base metric $g_{B,1}$ in this step.  Our warping function $f_2:B_2\to \dR^+$ is not a radial function everywhere, though one of our goals will be to make it radial on small radii.  We will want $f_2$ to satisfy the properties

\begin{align}
	\begin{cases}
		f_2=\hat\delta f_1&\text{ if }r\geq r_1\\
		f_2=\delta r^\alpha&\text{ if }r\leq r_2
	\end{cases}\, .
\end{align}

Due to the nonradial nature of $f_1$, there is some subtlety which makes a naive interpolation between $f_1$ and $\delta r^\alpha$ insufficient.  Morally, this is due to the uncontrolled positivity of $\Ric$, which can contribute very negative terms if altered in a careless manner.  The following will be the main constructive lemma for $f_2$ in this subsection: \\

\begin{lemma}\label{lem:glue-f-helper}
	Let $(B_2(p),g_B)$ and $f_1$ be as in Lemma \ref{l:step2:1:main}.  Then for each $0<\alpha<1$, $0 < \hat{\delta}< 1$, and $0 < r_2 < r_2(\alpha)$, there exists $f_2:B_2(p)\to \dR^+$ with $r_2<C(4)r_2 = r_2^+<r_1$ and $\delta=\delta(\,\hat{\delta}\|f\|_\infty \mid r_2)$ such that
\begin{enumerate}
\item $f_2 = \hat{\delta} f_1$ on the region $r_2^+\leq r\leq 2$ , \label{eq:s2b-1}
\item $f_2 = \delta r^\alpha$ on the region $r\leq r_2$ , \label{eq:s2b-2}
\item $-\frac{f_2''}{f_2} \geq \alpha r^{-2}$ on the region $r_2\leq r\leq r^+_2$ , \label{eq:s2b-3} 
\item $\alpha^{-1}|\nabla_{g_B} \ln f_2|_{g_B}^2, |\nabla^2_{g_B} \ln f_2|_{g_B} < C(4)\,\alpha\, r^{-2} $ on the region $r_2\leq r\leq r^+_2$, \label{eq:s2b-4}
\item $\|f_2\|_\infty \leq C(4)(\hat{\delta}\|f\|_\infty)^{1/2}$ on the region $r_2\leq r\leq r^+_2$, \label{eq:s2b-5} 
\item $f_2$ is smooth away from $r\in \{0, r_{2},r^+_2\}$, and $C^1$ everywhere except $\{p\}$ . \label{eq:s2b-6}
\end{enumerate}
\end{lemma}

\begin{remark}
	The $C^1$ nature of $f_2$ is due to $(3)$, where we force a definite amount of radial concavity throughout the interpolation region.  This will later be smoothed with a $C^1$ gluing lemma.
\end{remark}

We will wait until the end of the subsection to prove the above, which will require a bit of work.  Let us begin instead with how to use it in order to build our desired $g_2$.\\

\begin{lemma}\label{lem:glue-f}
	Let $(B_2(p)\times S^2,g_1)$ be as in Lemma \ref{l:step2:1:main}, and for $0<\alpha<\alpha(\epsilon)$, $r_2 = r_2(\lambda,\alpha,\eps)$, and $0<\hat{\delta}<\hat{\delta}(\lambda,\alpha,\eps, \|f\|_\infty, \hat{r})$, let $f_2$ be as in Lemma \ref{lem:glue-f-helper}. Then $g_2:= g_{B,1}+f_2^2 g_{S^2}$ satisfies $\Ric_{g_2}>\lambda - C(4)\eps$ away from $r \in [0,\hat{r}/2) \cup \{r_2, r_2^+\}$.
\end{lemma}

\begin{proof}[Proof of Lemma \ref{lem:glue-f} given Lemma \ref{lem:glue-f-helper}]
Let us apply Lemma \ref{lem:glue-f-helper}, to obtain $f_2$.  We will choose our constants $r_2, \alpha,\hat\delta$ later in the proof.  We begin with the Ricci curvature computation, where $g_B$ is the base metric for the {\it original} metric from Lemma \ref{lem:inductive_step2:2} :

\begin{align}
\label{eq:f-eps-prod}
(\Ric_{g_2})_{rr} &= (\Ric_{g_B})_{rr} - 2\frac{f_2''}{f_2}\,, \notag\\
(\Ric_{g_2})_{ir} &= (\Ric_{g_B})_{ir} - 2\frac{(\nabla^2_{g_B} f_2)_{ir}}{f_2}\,, \notag\\
(\Ric_{g_2})_{ij} &= (1 - \eps)^2(\Ric_{g_B})_{ij} + (1 - (1 - \eps)^2)(\Ric_{g_r})_{ij} -2\frac{(\nabla^2_{g_B} f_2)_{ij}}{f_2} + (1 - (1 - \eps)^2)\left(\frac{2}{r} (g_2)_{ij} + (r^2g_r')_{ij}\right)\frac{f_2'}{f_2}\,, \notag\\
(\Ric_{g_2})_{\alpha\beta} &= \left[\frac{1}{f_2^2} -\frac{1}{(1 - \eps)^2}\left(\frac{\Delta_{g_B} f_2}{f_2} + \frac{|\nabla^{g_B} f_2|^2_{g_B}}{f_2^2}\right) + \frac{1 - (1 - \eps)^2}{(1 - \eps)^2}\left( \frac{f_2''}{f_2} + \frac{(f_2')^2}{f_2^2} + \left(\frac{1}{2}\tr_{g_r}(g_r') + \frac{3}{r}\right)\frac{f_2'}{f_2} \right)\right](g_2)_{\alpha\beta}\,.
\end{align}

To estimate the Ricci curvature, we have three main regimes to study, namely $r\in (0,r_2)\cup (r_2,r^+_2)\cup (r^+_2,r_1)$.  The region $r \in (r_1,2)$ is covered by Remark \ref{rmk:mult-by-delta} and the previous Steps of the construction.  Let us begin with the region $r\in (r^+_2,r_1)$, and let us use Lemma \ref{l:step2:1:main}.  If we allow $\hat\delta<\hat\delta(\|f\|_\infty,r_2^+,|\lambda|)$ then we can use Lemma \ref{l:step2:1:main} to estimate
\begin{align}
	&\left.\Ric_{g_2}\right|_{TB_2(p)}>(\lambda - C(4)\,\eps)\, ,\notag\\
	&\left.\Ric_{g_2}\right|_{TS^2} > f_2^{-2} - C(4)\eta\, r^{-1} \geq \frac{1}{\hat\delta^2f^2}-\frac{C\eta}{r_2^+}>\lambda\, .
\end{align}
In particular, once $r_2^+$ has been fixed we may fix $\hat\delta$ sufficiently small to control the above.  Let us therefore focus on the more challenging situations: when $r\in (0,r_2)\cup(r_2,r^+_2)$.\\

To begin, let us record some basic estimates which will be used.  Recalling \eqref{e:step2:1:metric_decay_estimates}, we have that

\begin{equation}
(2 - C\eta r^2)g_r \leq \Ric_{g_r}\, ,\qquad\|g'_r\|_{L^\infty(S^3, g_r)}\leq C\eta\, r \, .
\end{equation}

In view of Lemma \ref{lem:glue-f-helper}, we have in the region $r_2 < r < r^+_2$ that

\begin{equation}\label{e:step2.2:f_estimates}
\alpha^{-1}|\nabla_{g_B} \ln f_2|_{g_B}^2, |\nabla^2_{g_B} \ln f_2|_{g_B} < C\frac{\alpha}{r^2}\,,\qquad \frac{\alpha}{2r^2} \leq -\frac{f''_2}{f_2} \, .
\end{equation}

Moreover, the same estimates hold when $0 < r < r_2$, because $f_2 = \delta r^\alpha$ in this region. Let us analyze each $B_2(p) \times S^2$ block of $\Ric_{g_2}$ according to \eqref{eq:f-eps-prod}.  For the $r,i,j,..$ block corresponding to the base $B_2(p)$ we can use the above estimates to get

\begin{align}
(\Ric_{g_2})_{rr} &> \frac{\alpha}{2r^2} -  C\eta \,,\notag\\
\big|(\Ric_{g_2})_{ir}\big| &< \left(C(\eta)\frac{\alpha}{r^2} + C\eta\right)\,,\notag\\
(\Ric_{g_2})_{ij} &> \left(\frac{(2 - C\eta r^2)\eps}{r^2} -C\left(\frac{\alpha}{r^2} + \eta \right)\right)(g_2)_{ij}\, .
\end{align}

If we require that $\alpha < \alpha(\eps)$ and $r_2^{+} < r_2^{+}(\alpha, \eps)$ are small enough, then

\begin{equation}
(\Ric_{g_2})_{rr} > \frac{\alpha}{4r^2}\,,\qquad (\Ric_{g_2})_{ij} > \frac{\eps}{r^2}(g_2)_{ij}\,,\qquad |(\Ric_{g_2})_{ir}| < C\frac{\alpha}{r^2}\,.\label{eq:f-Ric-pos-ests}
\end{equation}

If $\alpha < \alpha(\eps)$ is again small enough then this gives the estimate $\left.\Ric_{g_2}\right|_{TB_2(p)} > \frac{\eps}{2r^2}$ .  If we finally require that $r_2^+<r_2^+(\epsilon,|\lambda|)$ then this gives the required estimate for the $B_2$ block.  \\

For the $S^2$ block of $\Ric_{g_2}$, we argue separately for $r \in (\hat{r}/2,r_2)$ and $r \in (r_2,r_2^+)$. When $r\in (r_2,r_2^+)$ let us use \eqref{eq:f-eps-prod} and \eqref{e:step2.2:f_estimates} in order to estimate

\begin{equation}
(\Ric_{g_2})_{\alpha\beta} > \left(\frac{1}{\|f_2\|_\infty^2} - C\left(\frac{\alpha}{r_2^2} + \eps\eta\alpha\right)\right)(g_2)_{\alpha\beta}\,.
\end{equation}

Recalling the bound $\|f_2\|_\infty \leq C(\hat{\delta} \|f\|_\infty)^{1/2}$ from Lemma \ref{lem:glue-f-helper}, it is clear that a small enough choice of $\hat{\delta} < \hat{\delta}(\|f\|_\infty, \lambda, r_2)$ ensures that $(\Ric_{g_2})_{\alpha\beta} > \lambda (g_2)_{\alpha\beta}$ for $r\in (r_2,r_2^+)$. We can now focus on the region $r \in (\hat{r}/2,r_2)$, where

\begin{equation}
(\Ric_{g_2})_{\alpha\beta} > \left(\frac{1}{(\delta r_2^\alpha)^2} - C\frac{\alpha}{\hat{r}^2}\right)(g_2)_{\alpha\beta}\,.
\end{equation}

We see that for $\delta < \delta(\hat{r}, |\lambda|)$, we have the desired $(\Ric_{g_2})_{\alpha\beta} > \lambda(g_2)_{\alpha\beta}$ in this region, which would complete the proof if $\delta$ were sufficiently small. However, we have from Lemma \ref{lem:glue-f-helper} that $\delta = \delta(\hat{\delta}\|f\|_\infty \mid r_2)$, so it suffices to require that $\hat{\delta} < \hat{\delta}(\|f\|_\infty, r_2, \hat{r}, |\lambda|)$ to ensure that $\delta$ is sufficiently small.
\end{proof}

\vspace{.2cm}

\begin{proof}[Proof of Lemma \ref{lem:glue-f-helper}]
Before diving into the proof, we establish some new notation for the sake of legibility. We label

\begin{equation}
f_+(r,\omega) := \hat{\delta}f_1(r,\omega)\,,\qquad f_-(r,\omega) = f_-(r) := \delta r^\alpha\,,
\end{equation}

\noindent where $\delta$ will be specified later in the proof. Also, instead of referring to the radii $r_2$ and $r^+_2$ directly, it will be convenient to write the interval of interpolation $(r_2,r^+_2)$ in terms of a midpoint $r_m := (r_2 + r_2^+)/2$ and a radius $\rho r_m := (r_2^+ - r_2)/2$. In particular, the choice of radii $r_2,r_2^+$ from the statement of Lemma \ref{lem:glue-f-helper} will instead take the form of a choice of a constant $\rho$ and radius $r_m < r_m(\alpha)$. \\

The remaining proof has multiple steps, which we will break down into pieces: \\

{\bf (Locating the Intersection Set of $f_-$ and $f_+$):}  We first seek to set it up so that the intersection set $\{f_+ = f_-\} = \{(r,\omega) \mid f_+(r,\omega) = f_-(r)\}$ is approximately at our radius $r_m$, which requires selecting $\delta$ depending on $r_m$.  We will estimate the deviation of the intersection set from this radius.  These estimates will be used in the next steps of the proof.\\

Begin by observing that the $\nabla_{g_B} \ln f$ bound gives the estimate
\begin{equation}
\label{eq:f-plus-est}
|\ln f_+(r,\omega) - \ln f_+(p)| \leq \eta r\qquad  \implies \qquad f_+(r,\omega) \in [f_+(p)e^{-2\eta r_m}, f_+(p)e^{2\eta r_m}]\,, \qquad \forall r \leq 2r_m \,.
\end{equation}

One then checks under what conditions $f_-$ takes values in this same interval:

\begin{align}
\label{eq:f-minus-est}
&f_-(r) = \delta r^\alpha \in [f_+(p)e^{-2\eta r_m}, f_+(p)e^{2\eta r_m}] \notag\\
\iff\qquad &r \in [(\delta^{-1}f_+(p)e^{-2\eta r_m})^{1/\alpha},(\delta^{-1}f_+(p)e^{2\eta r_m})^{1/\alpha}] \,.
\end{align}

It is at this point that we make the choice $\delta := f_+(p)r_m^{-\alpha}e^{2\eta r_m}$. For this choice of $\delta$, we have $(\delta^{-1}f_+(p)e^{2\eta r_m})^{1/\alpha} = r_m$, and thus when \eqref{eq:f-minus-est} holds we have $r \leq 2r_m$.  Observe that we now have the relationships

\begin{equation}
 \begin{cases}
f_-(r)< f_+(p)e^{-2\eta r_m} \leq f_+(r,\omega),& r < r_m e^{-4\eta r_m/\alpha}\\
f_-(r)\in [f_+(p)e^{-2\eta r_m}, f_+(p)e^{2\eta r_m}],& r_me^{-4\eta r_m/\alpha}\leq r \leq r_m \\
f_-(r)> f_+(p)e^{2\eta r_m} \geq f_+(r,\omega),&  r_m < r \leq 2 r_m
\end{cases}\,.
\end{equation}

In particular, we have

\begin{equation}
\{f_- = f_+\} \cap B_{2r_m}(p) \subset A_{r_me^{-4\eta r_m/\alpha}, r_m}(p)\, .
\end{equation}

Let us additionally remark that $\{f_- = f_+\} \cap B_{2r_m}(p) \cap (\R{}_+ \times \{\omega\})$ is nonempty for each $\omega \in S^3$ by the intermediate value theorem. We may therefore define $g(\omega) \in [r_m e^{-4\eta r_m/\alpha}, r_m]$ for each $\omega \in S^3$ to be a radius such that $f_+(g(\omega),\omega) = f_-(g(\omega))$.  We can ensure that $(g(\omega),\omega)$ always lies in the interpolation region $r \in ((1 - \rho)r_m, (1 + \rho)r_m)$ by requiring that $r_m < r_m(\alpha)$ is small enough that $e^{-4\eta r_m/\alpha} > 1 - \rho$.  It is not necessary that $g$ be continuously defined.\\

{\bf (Definition of $f_2$ and $C^1$ Cubic Interpolation):} Our construction of $f_2$ will make use of the general notion of a $C^1$ cubic interpolation.  Namely, we will ask that $f_2$ be the uniquely defined $C^1$ function satisfying
\begin{align}
	\ln f_2(r,\omega) = \begin{cases}
		\ln f_-(r) &\text{ if } r\leq (1-\rho)r_m\, \\
		Q(r,\omega)=cubic\; polynomial\;in\; r &\text{ if } r\in  \big((1-\rho)r_m, (1+\rho)r_m\big)\, \\
		\ln f_+(r,\omega) &\text{ if } r\geq (1+\rho)r_m\, 
	\end{cases}\, .
\end{align}
 
We see from the above that $Q(r,\omega)$ is therefore well defined by the values of $f_{\pm}$ and $f'_{\pm}$ at the end points of the interval $[(1- \rho)r_m,(1+\rho)r_m]$ .  It will be helpful to write the form of $Q(r,\omega)$ explicitly by
\begin{align}\label{eq:cubic-int}
	Q(r,\omega) &:= \left[\frac{r - (1 - \rho)r_m}{2\rho r_m}\right]^2\left[\ln f_+\big((1+\rho)r_m,\omega\big) - \big((1+\rho)r_m - r\big)\Big((\ln f_+)'\big((1+\rho)r_m,\omega\big) - \frac{1}{\rho r_m}\ln f_+\big((1+\rho)r_m,\omega\big)\Big)\right] \notag\\
&+ \left[\frac{(1+\rho)r_m - r}{2\rho r_m}\right]^2\left[\ln f_-\big((1-\rho)r_m\big) + \big(r - (1-\rho)r_m\big)\Big((\ln f_-)'\big((1-\rho)r_m\big) + \frac{1}{\rho r_m}\ln f_-\big((1-\rho)r_m\big)\Big)\right]\,.
\end{align}

\vspace{.4cm}

{\bf (Concavity Estimates for $\ln f_2$):}  We can now estimate $(\ln f_2)''$ in the interpolation region $r \in ((1 - \rho)r_m, (1 + \rho)r_m)$.  Below, we use the mean value theorem to estimate the difference, at $g(\omega)$, between $r\mapsto \ln f_\pm(r,\omega)$ and its linearization centered at $r = (1 \pm \rho)r_m$:

\begin{align}
(\ln f_2)'' &= \frac{1}{2\rho r_m}\left((\ln f_+)'((1 + \rho)r_m,\omega) - (\ln f_-)'((1 - \rho)r_m)\right) \notag\\
&+ \frac{3(r - r_m)}{2\rho^3r_m^3}(g(\omega) - r_m)\left((\ln f_+)'((1 + \rho)r_m,\omega) - (\ln f_-)'((1 - \rho)r_m)\right)\notag\\ 
&+\left. \frac{3(r - r_m)}{2\rho^3r_m^3}\left[((1 + \rho)r_m - g(\omega)) (\ln f_+)'((1 + \rho)r_m,\omega) - \left(\ln f_+((1 + \rho)r_m,\omega) - \ln f_+(g(\omega),\omega)\right)\right]\,\right\} = O(\|\ln f_+''\|_{\infty})\notag\\
&+ \left.\frac{3(r - r_m)}{2\rho^3r_m^3}\left[(g(\omega) - (1 - \rho)r_m) (\ln f_-)'((1 - \rho)r_m) - \left(\ln f_-(g(\omega)) - \ln f_-((1 - \rho)r_m)\right)\right]\,\right\} = O(\|\ln f_-''\|_{\infty})\notag\\
&\leq \left(-\frac{1}{2\rho} + \frac{3}{2}\frac{|g(\omega) - r_m|}{\rho^2r_m} + \frac{C}{(1 - \rho)}\right)\frac{\alpha}{(1 - \rho)r_m^2} + \left(\frac{1}{2\rho} + \frac{3}{2}\frac{|g(\omega) - r_m|}{\rho^2 r_m} + Cr_m\right)\frac{\eta}{r_m}\,. \label{eq:f2-concavity}
\end{align}

The desired concavity should arise from the first term of \eqref{eq:f2-concavity}. We therefore pick the universal constant $\rho > 0$ small enough so that $-\frac{1}{2\rho} + \frac{C}{(1 - \rho)} < -\frac{4}{1 - \rho}$, and then $r_m \leq r_m(\alpha)$ small enough so that $|1 - e^{-4\eta r_m/\alpha}|\leq \frac{2\rho^2}{3(1 - \rho)}$. Since  $g(\omega) \in (r_m e^{-4\eta r_m/\alpha}, r_m)$, the constraint on $r_m$ implies that $|g(\omega) - r_m| \leq \frac{2\rho^2}{3(1 - \rho)}r_m$. Thus,

\begin{equation}
\left(-\frac{1}{2\rho} + \frac{3}{2}\frac{|g(\omega) - r_m|}{\rho^2r_m} + \frac{C}{(1 - \rho)}\right) \leq -\frac{3}{1 - \rho}\,.
\end{equation}

We also require that $r_m \leq r_m(\alpha)$ is small enough so that $\frac{\eta}{r_m} \leq (\frac{1}{2\rho} +\frac{1}{(1 - \rho)} + Cr_m)^{-1}\frac{\alpha}{(1 - \rho)^2r_m^2}$ in order to absorb the lower order second term of \eqref{eq:f2-concavity}. This yields the claimed concavity 

\begin{equation}
(\ln f_2)'' \leq -2\frac{\alpha}{[(1-  \rho)r_m]^2}\leq -2\frac{\alpha}{r^2}\,, \qquad \forall r \in ((1 - \rho)r_m,(1 + \rho)r_m)\,.
\end{equation}

\vspace{.4cm}

{\bf (Radial Zeroth and First Order Estimates on $f_2$):} We begin this step by obtaining a bound on $(\ln f_2)'$.  This allows us to turn the concavity estimates for $\ln f_2$ into concavity estimates for $f_2$ .  Additionally we will use the estimate on $(\ln f_2)'$ to obtain $L^\infty$ control for $f_2$ in the interpolation region. \\

The previous step implies that for fixed $\omega \in S^3$, $r\mapsto\ln f_2'(r,\omega)$ is decreasing from the value $\ln f_2'((1 - \rho)r_m,\omega) = \alpha [(1 - \rho)r_m]^{-1}$ at $r = (1 - \rho)r_m$. Thus,

\begin{align}
|(\ln f_2)'| &\leq \frac{\alpha}{(1 - \rho)r_m} \leq \alpha \frac{1 + \rho}{1 - \rho}r^{-1} \, .
\end{align}
Combining this with our concavity estimate for $\ln f_2$ gives
\begin{align}
\frac{f_2''}{f_2} &= (\ln f_2) '' + (\ln f_2')^2 \leq -2\frac{\alpha}{r^2} +  \left(\frac{1 + \rho}{1 - \rho}\right)^2 \frac{\alpha^2}{r^2} \leq -\alpha r^{-2}\,,
\end{align}

as long as $\alpha \leq \left(\frac{1 - \rho}{1 + \rho}\right)^2$.  Note that this proves Lemma \ref{lem:glue-f-helper}.\ref{eq:s2b-3}.  To obtain $L^\infty$ control on $f_2$, let us estimate $f_2$ on $((1 - \rho)r_m, (1 + \rho)r_m)$ by using the explicit formula \eqref{eq:cubic-int} for $Q$:

\begin{align}
Q(r,\omega) &\leq \left[\frac{r - (1 - \rho)r_m}{2\rho r_m}\right]^2\left[\ln f_+((1 + \rho)r_m,\omega) + \eta((1 + \rho)r_m - r)\right] \notag\\
&+ \left[\frac{(1 + \rho)r_m - r}{2\rho r_m}\right]^2\left[\ln f_-((1 - \rho)r_m) + \alpha\frac{(r - (1 - \rho)r_m)}{(1 - \rho)r_m}\right] \notag\\
&\leq \left(\left[\frac{r - (1 - \rho)r_m}{2\rho r_m}\right]^2 + \left[\frac{(1 + \rho)r_m - r}{2\rho r_m}\right]^2\right)\ln (\|f_+\|_\infty) + \frac{2\rho}{1 - \rho}\alpha + 2\rho r_m\eta\notag\\
&\leq \frac{\ln \|f_+\|_\infty}{2}+ C\alpha + Cr_m\eta \,.
\end{align}

Note that for the second inequality above, we used that $f_-((1 - \rho)r_m) \leq f_-(g(\omega)) = f_+(g(\omega))$. Since $f_-$ is increasing, this holds if $(1 - \rho)r_m \leq g(\omega)$, which itself will hold if $r_m < r_m(\alpha)$ is small enough because $g(\omega) \geq r_m e^{-4\eta r_m/\alpha}$. Taking $\exp$ on both sides proves Lemma \ref{lem:glue-f-helper}.\ref{eq:s2b-5}.\\

{\bf (Remaining Derivative Estimates):} The last remaining item is Lemma \ref{lem:glue-f-helper}.\ref{eq:s2b-4}, which we now turn to. The constraint $r_m \leq r_m(\alpha)$ is also finalized in this last step. \\

For the remaining computations, we pick coordinates $\partial_i$ on $S^3$ that are normal at $\omega \in S^3$ for the metric $g_r$, and assume that we are working in the region $r\in ((1 - \rho)r_m, (1 + \rho)r_m)$. Before launching into the remaining derivative estimates, we recall two basic bounds that we will need, both from Corollary \ref{c:cone_exponential}:

\begin{equation}
\|g_r'\|_{L^\infty(S^3,g_r)} \leq C(4) \eta r\,,\qquad \|g_r''\|_{L^\infty(S^3,g_r)}\leq C(4)\eta\,.
\end{equation}

We begin with the tangential derivatives of $f_2$. One notices the identity, using the notation \eqref{eq:cubic-int},

\begin{align}\label{e:step2:derivative_log_f:1}
	\partial_i \ln f_2 &= \left[\frac{r - (1 - \rho)r_m}{2\rho r_m}\right]^2\left[\partial_i\ln f_+\big((1+\rho)r_m,\omega\big) - \big((1+\rho)r_m - r\big)((\partial_i\ln f_+)'\big((1+\rho)r_m,\omega\big) - \frac{1}{\rho r_m}\partial_i\ln f_+\big((1+\rho)r_m,\omega\big)\right]\,.
\end{align}

What is important is that the right hand side is a linear combination of $\partial_i \ln f_+((1 + \rho)r_m,\omega)$ and $O(r)\partial_i \ln f_+'((1 + \rho)r_m,\omega)$ with uniformly bounded coefficients.   The first term $\partial_i \ln f_+((1 + \rho)r_m,\omega)$ is bounded by $\eta$, while  the second equals $O(r)(\partial_i \ln f_+') = O(r)(\nabla_{i,r}^2\ln f_+ + r^{-1}(\delta_i^k + rg_r^{kj}(g_r')_{ji}/2)\partial_k\ln f_+)$, which can be bounded by $C\eta$. Thus, for $r_m \leq r_m(\alpha)$, we can estimate

\begin{align}
|\nabla_i \ln f_2|_{r^2(g_r)_{ij}} \leq C(4)\eta \leq \frac{\alpha}{r}\,.
\end{align}

Similarly, $\partial_i \partial_j\ln f_2$ is a linear combination of $\partial_i \partial_j\ln f_+((1 + \rho)r_m,\omega)$ and $O(r)\partial_i\partial_j \ln f_+'((1 + \rho)r_m,\omega)$ with bounded coefficients, where

\begin{align}
\partial_i\partial_j \ln f_+ &= \nabla^2_{i,j}f_+ - \frac{1}{2}(2r(g_r)_{ij} + r^2(g_r')_{ij})\partial_r f_+\,, \notag\\
\partial_i\partial_j \ln f_+' &= \nabla^3_{r,i,j} f_+ - \frac{1}{2}(2(g_r)_{ij} + 4r(g_r')_{ij} + r^2(g_r'')_{ij})\partial_r f_+ -\frac{1}{2} (2r(g_r)_{ij} + r^2(g_r')_{ij})\partial^2_rf_+ \notag\\
&+ \frac{1}{2r}(2\delta_i^k + rg_r^{k\ell}(g_r')_{\ell i})\nabla^2_{k,j}f_+ + \frac{1}{2r}(2\delta_j^k + rg_r^{k\ell}(g_r')_{\ell j}))\nabla^2_{k,i}f_+\,. 
\end{align}

The highest order term is $-(g_r)_{ij}\partial_r f_+$ from the second equation, and can be bounded by $\eta r^{-2}$.  The other terms can be estimated by $C\eta r^{-1}$. Using this and the previously obtained bound on $\ln f_2'$, and again requiring that $r_m \leq r_m(\alpha)$ we get

\begin{align}
|\nabla^2_{ij}\ln f_2|_{r^2(g_r)_{ij}} &\leq |\partial_i\partial_j\ln f_2|_{r^2(g_r)_{ij}} + \frac{1}{2}|(2r(g_r)_{ij} + r^2(g_r')_{ij})\ln f_2'|_{r^2(g_r)_{ij}} \notag\\
&\leq C\frac{\eta}{r} + C\frac{\alpha}{r^2} \leq C\frac{\alpha}{r^2}\,.
\end{align}

Finally we repeat this strategy once more and observe as in \eqref{e:step2:derivative_log_f:1} that $\partial_i\partial_r \ln f_2 $ is a linear combination of $\partial_i \ln f_+((1 + \rho)r_m,\omega)$ and $O(r)\partial_i \ln f_+'((1 + \rho)r_m,\omega)$ with coefficients that are bounded times a factor  $\leq r_m^{-1}$. We have already bounded these two expressions by $C\eta$, so for $r_m \leq r_m(\alpha)$ we get

\begin{align}
|\nabla^2_{i,r}\ln f_2|_{r^2(g_r)_{ij}} &\leq |\partial_i\partial_r\ln f_2|_{r^2(g_r)_{ij}} + \frac{1}{2r}|(2\delta_i^k + rg_r^{k\ell}(g_r')_{\ell i})\partial_k\ln f_2|_{r^2(g_r)_{ij}} \notag\\
&\leq C(4) \eta r_m^{-1} + C(4)\eta r^{-1} \leq C(4)\eta\frac{\alpha}{r^2}\,,
\end{align}
which finishes the proof of Lemma \ref{lem:glue-f-helper}.

\end{proof}

\subsection{Step 2.3: Interpolating the Warped Cone to a Homogeneous Cone}\label{ss:step2:3}

We ended the last subsection having constructed a metric
\begin{align}
	g_{2} = g_{B,2} + f_2^2 g_{S^2}\, .
\end{align}
This metric has the property that for $r\leq r_2$ we can write
\begin{align}
	g_{B,2} &= dr^2 +(1-\epsilon)^2 r^2 g_r\, ,\notag\\
	f_2 &= f_2(r)= \delta r^\alpha\, ,
\end{align}
where recall $g_B = dr^2 + r^2 g_r$ was the original smooth metric with regularity scale $r_x>2$ , see \eqref{e:step2:regularity_scale}.  Our final goal in this section is to interpolate $g_r$ with the unit sphere metric $g_{S^3}$.  In this way we will build a metric

\begin{align}\label{e:step2:3:metric_ansatz}
	g_{3} := g_{B,3} + f_3^2 g_{S^2}\, ,
\end{align}

where $f_3 := f_2$ and $g_{B,3} = dr^2 +(1-\epsilon)^2r^2 g_{r,3}$ with
\begin{align}\label{e:step2:3:metric_ansatz:2}
	g_{r,3}:= \begin{cases}
 	g_r&\text{ if } r\geq 2r_3\, ,\\
 	\xi(r)g_{r}+(1-\xi(r))g_{S^3}&\text{ if } r\in [r_3,2r_3]\\
 	g_{S^3}&\text{ if } r\leq r_3
 \end{cases}\, .
\end{align}

Here, $\xi(r)$ is a smooth cutoff on $[r_3,2r_3]$ satisfying the $r_3$-independent bounds

\begin{equation}
|\xi'(r)| \leq Cr^{-1}\,,\qquad |\xi''(r)| \leq Cr^{-2} \,.
\end{equation}

We will show that for $r_3$ sufficiently small, the metric $g_3$ will have the appropriate Ricci curvature bounds:

\begin{lemma}
	\label{lem:mod-x-sec}
	Let $g_3$ be defined as in \eqref{e:step2:3:metric_ansatz} and \eqref{e:step2:3:metric_ansatz:2} under the assumptions of Lemma \ref{lem:inductive_step2:2}.  Then for $r_3\leq r_3(\alpha,\lambda,\epsilon)$ we have that $\Ric_{g_3}>\lambda - C(4)\eps$ on $\big(B_2(p) \setminus B_{\hat{r}/2}(p)\big)\times S^2$.
\end{lemma}

\begin{proof}
	For radii $r > 2r_3$, the Ricci lower bound follows from Lemma \ref{lem:glue-f} and the fact that $g_2 = g_3$ in this region. Thus, we assume for the rest of the proof that $r \leq 2r_3$. The main basic estimates we will need are from Corollary \ref{c:cone_exponential}, namely
	
\begin{align}
\label{eq:step-3-met-ests}
\|g_r - g_{S^3}\|_{C^2(S^3, g_r)} &\leq C\eta r^2 \,, \notag\\
\|g_r'\|_{C^1(S^3, g_r)} &\leq C\eta r \,, \notag\\
\|g_r''\|_{L^\infty(S^3, g_r)} &\leq C\eta\,.
\end{align}

Let us first apply the above in order to see the estimates

\begin{align}
\label{eq:step-3-int-ests}
\|g_{r,3} - g_{S^3}\|_{C^2(S^3,g_{r,3})} &\leq C\eta r^2\,,\notag\\
\|g_{r,3}'\|_{C^1(S^3,g_{r,3})} = \|\xi'(g_r - g_{S^3}) + \xi g_r'\|_{C^1(S^3,g_{r,3})} &\leq C \eta r\,, \notag \\
\|g_{r,3}''\|_{L^\infty(S^3,g_{r,3})} = \|\xi''(g_r - g_{S^3}) + 2\xi' g_r' + \xi g_r''\|_{L^\infty(S^3,g_{r,3})} &\leq  C \eta \,.
\end{align}

We are now in a position to estimate $\Ric_{g_3}$, which we compute as follows:

\begin{align}
(\Ric_{g_3})_{rr} &= -\frac{1}{2}\tr_{g_{r,3}}(g_{r,3}'') - \frac{1}{r}\tr_{g_{r,3}}(g_{r,3}') + \frac{1}{4}|g_{r,3}'|_{g_{r,3}}^2 - 2\frac{\alpha(\alpha - 1)}{r^2}\,, \notag\\
(\Ric_{g_3})_{ir} &= \frac{1}{2}(\partial_i\tr_{g_{r,3}}(g_{r,3}') - (\tr_{g_{r,3}}^{1,2}(\nabla g_{r,3}'))_i)\,,\notag\\
(\Ric_{g_3})_{ij} &= (1 - (1 - \eps)^2)(\Ric_{g_{S^3}})_{ij} + (\Ric_{g_{r,3}} - \Ric_{g_{S^3}})_{ij} - \frac{(1 - \eps)^2}{2}(r^2g_{r,3}'')_{ij} \notag\\
&- (1 - \eps)^2\left(\frac{3}{2r} + \frac{1}{4}\tr_{g_{r,3}}(g_{r,3}')\right)(r^2g_{r,3}')_{ij} + \frac{(1 - \eps)^2}{2}(rg_{r,3}')_{ik}g_{r,3}^{k\ell}(rg_{r,3}')_{\ell j} + \frac{1}{2r}\tr_{g_{r,3}}(g_{r,3}')(g_{3})_{ij} \notag\\
&-\alpha\left(\frac{2}{r^2}(g_3)_{ij} + (1 - \eps)^2(rg_{r,3}')_{ij}\right) + 2((g_{S^3})_{ij} - (g_{r,3})_{ij})\,, \notag\\
(\Ric_{g_3})_{\alpha\beta} &= \left[\frac{1}{\delta^2 r^{2\alpha}} - \frac{\alpha}{r^2}\left(2(1 + \alpha) + \frac{1}{2}\tr_{g_{r,3}}(rg_{r,3}')\right)\right](g_3)_{\alpha\beta} \,.
\end{align}

We may now use the bounds \eqref{eq:step-3-int-ests} above to estimate the components of $\Ric_{g_3}$ by

\begin{align}
(\Ric_{g_3})_{rr} &> 2\frac{\alpha(1 - \alpha)}{r^2} - C(4)\eta\,, \notag\\
|(\Ric_{g_3})_{ir}|_{(g_3)_{ij}} &< C(4)\eta\,, \notag\\
(\Ric_{g_3})_{ij} &> \left(2\frac{\eps}{r^2} - C\left(\frac{\alpha}{r^2} + C(4)\eta\right)\right)(g_3)_{ij}\,, \notag\\
(\Ric_{g_3})_{\alpha\beta} &> \left(\frac{1}{\delta^2 r^{2\alpha}} - C\left(\frac{\alpha}{r^2} + C(4)\eta\right)\right)(g_3)_{\alpha\beta}\,.
\end{align}

We see from these estimates that requiring $\alpha \leq C\eps$, $r_3 \leq r_3(\alpha,\eps, \lambda)$, and $\delta < \delta(\hat{r}, |\lambda|)$ small enough guarantees that $\Ric_{g_3} > \lambda$ for the range $\hat{r}/2\leq r \leq 2r_3$ under consideration. Note that without loss of generality these requirements on $\alpha$ and $\delta$ (by way of $\hat{\delta}$) held when they were chosen in Lemma \ref{lem:glue-f}.

\end{proof}

\subsection{Proof of Lemma \ref{lem:inductive_step2:2}}\label{ss:step2:4}

By construction, the metric $g_3$ with inner radius $\hat{r}$ subject to $\hat{r} < r_3 = r_3(\alpha,\lambda,\eps)$ from Lemma \ref{lem:mod-x-sec} satisfies $\Ric_{\hat{g}} > \lambda - C(4)\eps$  on $(B_2(p)\setminus B_{\hat{r}/2}(p))\times S^2$ and all the conditions of Lemma \ref{lem:inductive_step2:2}, except that it is not smooth.  This metric is globally $C^1$ away from $p$, but fails to be smooth at $r \in \{0, \,r_2,\,r^+_2\}$. The singularity at $0$ is as described in the statement of Lemma \ref{lem:inductive_step2:2}, and the other two radii can be smoothed within the class of warped product metrics to a metric $\hat{g} = \hat{g_B} + \hat{f}^2g_{S^2}$ while preserving the $\Ric$ bound by the $C^1$ gluing of Lemma \ref{l:C1_warped_gluing}.\\

The last remaining statement of Lemma \ref{lem:inductive_step2:2} to be proved is that the identity map $\Id:(B_2(p),\hat{g}_B)\to (B_2(p),g_B)$ is $(1+2\epsilon)$-bi-Lipschitz.  As the smooth metric $\hat{g}$ can be made $C^1$ close to $g_3$, we can prove a slightly stronger estimate directly for $g_3$ and the result will follow for our final smooth metric as in Lemma \ref{l:C1_warped_gluing}. \\

 The bi-Lipschitz condition will follow from the following metric comparisons for different ranges of $r$, which themselves follow by construction and the estimate $\|g_{r,3} - g_{r}\|_{C^0(S^3,g_{r})} \leq C\eta r^2$:

\begin{equation}
\begin{drcases}
g_{B}& \\
(1 - \eps)^2g_B& \\
(1- \eps)^2(1 - C(4)\eta r^2) g_B&
\end{drcases} 
\leq g_{B,3} \leq 
\begin{cases}
g_B\,,& r \in (1,2] \\
g_B\,,& r \in (2r_3,1]\\
(1 + C(4)\eta r^2)g_B\,,& r \in (0, 2r_3]
\end{cases}\,.
\end{equation}

By again requiring that $r_3 \leq r_3(\eps)$ in Lemma \ref{lem:mod-x-sec}, we can simplify the above by replacing all lower and upper bounds by $(1 \pm \frac{3}{2}\eps)^2g_B$ respectively.  Estimating the bi-Lipschitz constant of $\phi$ then becomes a matter of unraveling definitions as

\begin{equation}
(1 - \frac{3}{2}\eps)|D\phi(v)|_{g_B} = (1- \frac{3}{2}\eps)|v|_{g_B} \leq |v|_{g_{B,3}} \leq (1 + \frac{3}{2}\eps)|D\phi(v)|_{g_B}\,, \qquad \forall v\in (TB_2(p), g_{B,3})\, ,
\end{equation}

which completes the proof. $\qed$
\vspace{.5cm}

\bibliographystyle{aomalpha}
\bibliography{HNW_Nonmanifold_Refs}

\end{document}